\documentclass[a4paper,12pt,reqno]{amsart}
\usepackage{amsmath, amscd, amssymb}

\usepackage{multirow}

\newtheorem{lem}{Lemma}[section]
\newtheorem{thm}[lem]{Theorem}
\newtheorem{pro}[lem]{Proposition}  

\newtheorem{exa}[lem]{Example}

\numberwithin{equation}{section}

\newcommand{\ww}{{\mathbf{w}}}
\newcommand{\xx}{{\mathbf{x}}}

\newcommand{\inv}{{\mathsf{inv}}}
\newcommand{\cyc}{{\mathsf{cyc}}}

\newcommand{\wt}{{\mathsf{wt}}}
\newcommand{\exc}{{\mathsf{exc}}}
\newcommand{\lucky}{{\mathsf{lucky}}}
\newcommand{\reluc}{{\mathsf{reluc}}}
\newcommand{\bad}{{\mathsf{bad}}}
\newcommand{\npc}{{\mathsf{npc}}}

\newcommand{\PF}{{\mathrm{PF}}}
\newcommand{\NC}{{\mathrm{NC}}}

\newcommand{\PP}{{\mathcal{P}}}

\newcommand{\X}{{\mathcal{X}}}

\newcommand{\F}{{\mathcal{F}}}

\title[$q$-Counting of Noncrossing Chains]{On $q$-Counting of Noncrossing Chains and \\Parking Functions}

\author[Y.-J. Cheng]{Yen-Jen Cheng}
\address{Department of Applied Mathematics,  National Yang Ming Chiao Tung University, Hsinchu 300093, Taiwan, ROC}
\email{yjc7755@nycu.edu.tw}

\author[S.-P. Eu]{Sen-Peng Eu}
\address{Department of Mathematics, National Taiwan Normal University, Taipei 116325, and Chinese Air Force Academy, Kaohsiung 820009, Taiwan, ROC}
\email{speu@math.ntnu.edu.tw}

\author[T.-S. Fu]{Tung-Shan Fu}
\address{Department of Applied Mathematics, National Pingtung University, Pingtung 900391, Taiwan, ROC}
\email{tsfu@mail.nptu.edu.tw}

\author[J.-C. Yao]{Jyun-Cheng Yao}
\address{Department of Mathematics, National Taiwan Normal University, Taipei 116325, Taiwan, ROC}
\email{0955526287er@gmail.com}

\begin{document}


\begin{abstract} 
For a finite Coxeter group $W$, Josuat-Verg\`es derived a $q$-polynomial counting the maximal chains in the lattice of noncrossing partitions of $W$ by weighting some of the covering relations, which we call bad edges, in these chains with a parameter $q$. We study the connection of these weighted chains with parking functions of type $A$ ($B$, respectively) from the perspective of the $q$-polynomial. The $q$-polynomial turns out to be the generating function for parking functions (of either type) with respect to the number of cars that do not park at their preferred spaces.  In either case, we present a bijective result that carries bad edges to unlucky cars while preserving their relative order. Using this, we give an interpretation of the $\gamma$-positivity of the $q$-polynomial in the case that $W$ is the hyperoctahedral group.
\end{abstract}

\maketitle

\section{Introduction} 
For a finite Coxeter group $W$ of rank $n$, Josuat-Verg\`es \cite{JV} derived a polynomial $M(W,q)$ (as shown in Theorem \ref{thm:M-polynomial})
which enumerates the maximal chains in the lattice of noncrossing partitions of $W$ by weighting some of the covering relations, called \emph{bad edges} (as defined below), in these chains with a parameter $q$. 
When the polynomial specialized at $q=1$, the value in the case of symmetric group $\mathfrak{S}_{n+1}$ (hyperoctahedral group $B_n$, respectively) is $(n+1)^{n-1}$ ($n^n$, respectively), which is known to be the number of parking functions of type $A$ ($B$, respectively) of size $n$. The polynomial turns out to be the generating function which counts parking functions (of either type) by the number of cars that do not park at their preferred spaces. Biane \cite{Biane} established a bijection between the maximal chains in the lattice of noncrossing partitions of $\mathfrak{S}_{n+1}$ ($B_n$, respectively) and parking functions of type $A$ ($B$, respectively). However, these two bijections are not established with respect to the polynomial. In either case, we come up with a bijection that not only carries the bad edges of a maximal chain to the unlucky cars of the corresponding parking function but also preserves their relative order (Theorems \ref{thm:type-A-PF-NC} and \ref{thm:type-B-PF-NC}). A highlight of our results is that in the case of $W=B_n$ we prove the $\gamma$-positivity of the polynomial, in terms of words over an alphabet with $n$ letters (Theorem \ref{thm:gamma-positive-for-words}), and give interpretations of the $\gamma$-coefficients using parking functions and the weighted chains (Theorem \ref{thm:gamma-positive}).

\subsection{Noncrossing lattices of $W$}
Let $S$ be the set of simple generators of $W$. Let $T=\{\omega s\omega^{-1} : \omega\in W, s\in S\}$, whose elements are the \emph{reflections} in $W$ with respect to roots. For any $\omega\in W$, the \emph{length} $\ell_S(\omega)$ of $\omega$ is the minimum integer $r$ such that there exists an expression $\omega=s_1s_2\cdots s_r$ with $s_i\in S$. Write $\omega\rightarrow \omega'$ if $\omega'=\omega t$ for some $t\in T$ with $\ell_S(\omega)<\ell_S(\omega')$. 
The (\emph{strong}) \emph{Bruhat order} $\le_B$ on $W$ is defined by setting $\omega <_B \omega'$ if there is a sequence $\omega= \omega_0 \rightarrow \omega_1 \rightarrow \cdots \rightarrow \omega_m=\omega'$.
On the other hand, the \emph{reflection length} $\ell_T(\omega)$ of $\omega$ is the minimum integer $r$ such that there exists an expression  $\omega=t_1t_2\cdots t_r$ with $t_i\in T$. The \emph{absolute order} $\le$ on $W$ is defined by setting $\omega < \omega' $ if $\ell_T(\omega')=\ell_T(\omega)+\ell_T(\omega^{-1}\omega')$.

Relative to a Coxeter element $c\in W$, the $c$-\emph{noncrossing partition lattice} $\NC(W,c)$ is defined as the interval in the group $W$ between the identity element $e$ and the Coxeter element $c$ under the absolute order (e.g. \cite{Bessis}, \cite{BW-02}, \cite{BW}),
\begin{equation} 
\NC(W,c):=\{\omega\in W : e\le \omega\le c\}.
\end{equation}
The structure of $\NC(W,c)$ is independent of $c$ (e.g. \cite{RRS}). Deligne \cite{Deligne} obtained the formula
\begin{equation} \label{eqn:Deligne}
\frac{n! h^n}{|W|},
\end{equation}
where $h$ is the Coxeter number of $W$, counting the number of factorizations of $c$ into $n$ reflections. Chapoton \cite[Proposition 9]{Chapoton} rediscovered the expression (\ref{eqn:Deligne}) by counting the maximal chains in the lattice $\NC(W,c)$. Recently, Chapuy and Douvropoulos \cite{Chapuy-Douvro} gave a case-free Coxeter-theoretic derivation of this chain number formula.  

An \emph{edge} of $\NC(W,c)$ is a pair $(\mu,\nu)$ of elements in $W$ such that $\nu$ \emph{covers} $\mu$ (i.e., $\mu< \nu$ and no element $\omega$ with $\mu< \omega< \nu$). An edge $(\mu,\nu)$ is said to be \emph{good} (\emph{bad}, respectively) if $\mu<_B \nu$ ($\mu>_B \nu$, respectively). The notion of good and bad edges comes from \cite[Definition~3.1]{JV}; see also \cite[Definition~4.10]{BJ}.

Given a maximal chain $\ww: e=\omega_0<\omega_1<\cdots<\omega_n=c$ in $\NC(W,c)$, let $\bad(\ww)$ denote the number of bad edges in $\ww$, i.e., $\bad(\ww):=\#\{(\omega_{i-1},\omega_i) : \omega_{i-1}>_B \omega_i, 1\le i\le n\}$.  Define
\begin{equation} \label{eqn:gf-mc-bad}
M(W,q):=\sum_{\ww} q^{\bad(\ww)},
\end{equation}
summed over all maximal chains $\ww$ in $\NC(W,c)$. As a refinement of (\ref{eqn:Deligne}), Josuat-Verg\`es \cite[Proposition~3.6]{JV} obtained the following result.

\begin{thm}[Josuat-Verg\`es] \label{thm:M-polynomial}
We have
\begin{equation*}
M(W,q)=\frac{n!}{|W|}\prod_{i=1}^{n} \big(d_i+q(h-d_i)\big),
\end{equation*}
where the integers $d_1,d_2,\dots,d_n$ are the fundamental degrees of $W$.
\end{thm}

We shall study the connection of parking functions of type $A$ ($B$, respectively) with the maximal chains in $\NC(W,c)$ by the polynomial $M(\mathfrak{S}_{n+1},q)$ ($M(B_n,q)$, respectively).

\subsection{Parking functions of type $A$} Consider a one-way street with $n$ parking spaces numbered from 1 to $n$ and a line of $n$ cars $C_1,\dots,C_n$ with preferred parking spaces $a_1,\dots, a_n$, accordingly.
The cars arrive sequentially and each car attempts to park at its preferred space. If that space is occupied, the car moves on and parks at the next available space. 
If there is no such parking space, the car leaves the street without parking.
The preference list $\alpha=(a_1,\dots, a_n)$ is called a \emph{parking function} (of type $A$) of \emph{size} $n$ if all cars park successfully. For each $i$, assume that the $i$th car $C_i$ parks at space $p_i$. The $n$-tuple $\pi=(p_1,\dots, p_n)$ is called the \emph{outcome} of $\alpha$. A car $C_i$ is said to be \emph{lucky} (\emph{reluctant}, respectively) if $p_i=a_i$ ($p_i\neq a_i$, respectively). Let $\lucky(\alpha)$ ($\reluc(\alpha)$, respectively) denote the number of lucky (reluctant, respectively) cars of $\alpha$. Note that $\lucky(\alpha)+\reluc(\alpha)=n$. The $\lucky$ statistic for parking functions was studied by Gessel and Seo \cite{GS}. We refer the readers to \cite{Yan} for a survey on various aspects of parking functions.

Let $\PF^A_n$ denote the set of parking functions of size $n$. In the case of $W=\mathfrak{S}_{n+1}$, we have 
\begin{equation} \label{eqn:A-polynomial}
M(\mathfrak{S}_{n+1},q)=\prod_{k=1}^{n-1} \big(n-k+1+kq\big),
\end{equation} 
which coincides with the generating function for parking functions with respect to their reluctant cars (cf. \cite[Theorem~10.1]{GS}), i.e.,
\begin{equation}  \label{eqn:PF^A-generating-poly}
M(\mathfrak{S}_{n+1},q)=\sum_{\alpha\in\PF^A_n} q^{\reluc(\alpha)}.
\end{equation}
Biane \cite{Biane} established a bijection between the maximal chains $\ww$ in $\NC(\mathfrak{S}_{n+1},c)$ and the parking functions $\alpha$ in $\PF^A_n$. This bijection turns out to be equivalent to a result by Stanley \cite{Stanley}. However, this bijection does not carry the bad edges in $\ww$ to the reluctant cars of $\alpha$. One of our main results is the following bijection.

\begin{thm} \label{thm:type-A-PF-NC} There is a bijection $\varphi_A:\alpha\mapsto \ww:e=\omega_0<\omega_1<\cdots<\omega_n=c$ between the parking functions in $\PF^A_n$ and the maximal chains in $\NC(\mathfrak{S}_{n+1},c)$ such that the $i$th car of $\alpha$ is reluctant if and only if $(\omega_{i-1},\omega_i)$ is a bad edge.
\end{thm}

\subsection{Parking functions of type $B$}
Now, arrange the $n$ parking spaces in a circle. There are still $n$ cars with preferred parking spaces $a_1,\dots, a_n$, accordingly. Starting just before space 1, each car in turn moves around the circle to park at its preferred space if possible; otherwise, it moves on to the next available space. Each car will eventually park. An $n$-tuple $\alpha=(a_1,\dots, a_n)$ of positive integers is a \emph{parking function of type $B$} of \emph{size} $n$ if $a_i\le n$ for all $i$. The outcome of $\alpha$ and the statistics $\lucky(\alpha)$ and $\reluc(\alpha)$ are defined in the same manner as that of a type-$A$ parking function.
Let $\PF^B_n$ denote the set of parking functions of type $B$ of size $n$.

Let $[n]:=\{1,\dots,n\}$ and let $[\pm n]:=\{1,2,\dots,n,-1,-2,\dots,-n\}$. A \emph{signed permutation} of the set $[\pm n]$ is a bijection $\omega$ of the set onto itself such that $\omega(-i)=-\omega(i)$ for all $i\in [n]$.
The \emph{hyperoctahedral group} $B_n$ is the group of signed permutations of $[\pm n]$.
In the case of $W=B_n$, we have 
\begin{equation} \label{eqn:B-polynomial}
M(B_n,q)=\prod_{k=0}^{n-1} \big(n-k+kq\big).
\end{equation}
It is easy to see that this polynomial is the generating function for parking functions of type $B$ with respect to their reluctant cars, i.e., 
\begin{equation}  \label{eqn:PF^B-generating-poly}
M(B_n,q)=\sum_{\alpha\in\PF^B_n} q^{\reluc(\alpha)}.
\end{equation}
Biane \cite{Biane} gave an edge labeling of $\NC(B_n,c)$ that establishes a bijection between the maximal chains $\ww$ in $\NC(B_n,c)$ and the parking functions $\alpha$ in $\PF^B_n$. However, this bijection is not in respect of the statistics $\bad(\ww)$ and $\reluc(\alpha)$, either. We present a bijection that carries the bad edges in $\ww$ to the reluctant cars of $\alpha$ while preserving their relative order.

\begin{thm} \label{thm:type-B-PF-NC} There is a bijection $\varphi_B:\alpha\mapsto \ww:e=\omega_0<\omega_1<\cdots<\omega_n=c$ between the parking functions in $\PF^B_n$ and the maximal chains in $\NC(B_n,c)$ such that the $i$th car of $\alpha$ is reluctant if and only if $(\omega_{i-1},\omega_i)$ is a bad edge.
\end{thm}

\medskip
\subsection{Gamma-positivity of $M(B_n,q)$}
A polynomial $C(q)=c_0+c_1q+\cdots+c_nq^n$ with palindromic coefficients (i.e., $c_{n-i}=c_i$) can be written as a sum of polynomials of the form $q^j(1+q)^{n-2j}$, 
\[
C(q)=\sum_{j=0}^{\lfloor n/2\rfloor} \gamma_{j}\, q^j(1+q)^{n-2j}.
\]
The palindromic polynomial $C(q)$ is said to be $\gamma$-\emph{positive} if $\gamma_j\ge 0$ for all $j$. 
Several of the initial polynomials of $M(B_n,q)$ are listed below:
\begin{align*}
M(B_1,q) &= 1,\\
M(B_2,q) &= 2+2q,\\
M(B_3,q) &= 6+15q+6q^2, \\
M(B_4,q) &= 24+104q+104q^2+24q^3, \\
M(B_5,q) &= 120+770q+1345q^2+770q^3+120q^4.
\end{align*}
These terms appear in an OEIS sequence \cite[A071208]{oeis}. Notice that these palindromic polynomials can be expanded as follows.
\begin{align*}
M(B_1,q) &= 1,\\
M(B_2,q) &= 2(1+q),\\
M(B_3,q) &= 6(1+q)^2+3q, \\
M(B_4,q) &= 24(1+q)^3+32q(1+q), \\
M(B_5,q) &= 120(1+q)^4+290q(1+q)^2+45q^2.
\end{align*}

For $0\le k\le \lfloor\frac{n-1}{2}\rfloor$, let $B(n,k)$ be the set of $k$-tuples $(b_1,b_2,\dots,b_k)$ of integers such that $1<b_1<b_2<\cdots<b_k<n$ and $b_{i+1}-b_i>1$ for each $i\in [k-1]$. We present the following interpretation of $\gamma$-positivity of the polynomial $M(B_n,q)$.

\medskip
\begin{thm} \label{thm:gamma-positive} For all $n\ge 1$, we have
\begin{equation*}
M(B_n,q) = \sum_{j=0}^{\lfloor (n-1)/2\rfloor} \gamma_{n,j} q^j(1+q)^{n-1-2j},
\end{equation*}
where $\gamma_{n,j}$ counts the following objects:
\begin{enumerate} 
\item the parking functions  $\alpha$ in $\PF^B_n$ with $\reluc(\alpha)=j$ such that there are no two consecutive reluctant cars, and the last car of $\alpha$ is lucky;
\item the maximal chains $\ww:e=\omega_0<\omega_1<\cdots<\omega_n=c$ in $\NC(B_n,c)$ with $\bad(\ww)=j$ such that there are no two consecutive bad edges, and $(\omega_{n-1},\omega_n)$ is a good edge.
\end{enumerate}
Moreover, $\gamma_{n,j}$ can be calculated by
\begin{equation} \label{eqn:calculation}
\gamma_{n,j}=n!\cdot\sum_{(b_1,b_2,\cdots,b_j)\in B(n,j)} \prod_{i=1}^{j} \left(\frac{n-b_i}{b_i}\right).
\end{equation}
\end{thm}

\medskip
The rest of this paper is organized as follows. We establish the bijections in Theorems \ref{thm:type-A-PF-NC} and \ref{thm:type-B-PF-NC} in Sections 2 and 3, respectively. To prove the $\gamma$-positivity of $M(B_n,q)$, we prepare a bijective result, with words on an alphabet with $n$ letters, for the enumeration of $\gamma$-coefficients in Section 4. The proof of Theorem \ref{thm:gamma-positive} is given in Section 5.

\section{Noncrossing partitions and parking functions of type $A$}
\subsection{Noncrossing partitions of type $A$}
A \emph{noncrossing partition} of $[n]$ is a set partition with the condition that there are no $a,b,c,d$ with $1\le a<b<c<d\le n$, $\{a,c\}$ contained in some block of the partition, and $\{b,d\}$ contained in some other block. The set of noncrossing partitions of $[n]$ forms a lattice $\NC_{n}$, ordered by reverse refinement, with maximum element $\{\{1,2,\dots,n\}\}$ and minimum element $\{\{1\},\{2\},\dots,\{n\}\}$.

Fix the symmetric group $\mathfrak{S}_{n+1}$ on $[n+1]$ and the Coxeter element $c=(1,2,\dots,n+1)$. If we associate every permutation $\sigma\in\mathfrak{S}_{n+1}$ with the partition of $[n+1]$ given by the cycle structure of $\sigma$, this defines a poset isomorphism from $\NC(\mathfrak{S}_{n+1},c)$ to $\NC_{n+1}$ \cite{Biane97}. The reflections of $\mathfrak{S}_{n+1}$ are the transpositions, i.e., $T=\{(i,j) : 1\le i<j\le n+1\}$. A maximal chain $\ww: e=\omega_0<\omega_1<\cdots<\omega_n=c$ in $\NC(\mathfrak{S}_{n+1},c)$ corresponds to a factorization $t_1t_2\cdots t_n$ of $c$ into $n$ reflections, where $\omega^{-1}_{i-1}\omega_i=t_i\in T$. 

Suppose $\delta=(d_1,\dots,d_m)$ is a cycle of $\sigma$ and $t$ is a reflection with $t\le\delta$ (in absolute order), say $t=(d_i,d_j)$, where  $i<j$. If we perform $t$ and then $\sigma$, observe that $\delta$ is split into two cycles $(d_1,\dots,d_i,d_{j+1},\dots,d_m)$ and $(d_{i+1},\dots,d_j)$ of $\sigma t$.  Let $\cyc(\sigma)$ denote the number of \emph{cycles} in  $\sigma$. Notice that (e.g. \cite[Lemma~4.1.5]{Armstrong})
\begin{equation} \label{eqn:n-cyc-A}
\ell_T(\sigma) =n+1-\cyc(\sigma).
\end{equation}

Let $\sigma=\sigma(1)\sigma(2)\cdots\sigma(n+1)$. An \emph{inversion} of $\sigma$ is a pair $(\sigma(i),\sigma(j))$ such that $\sigma(i)>\sigma(j)$, $1\le i<j\le n+1$. The \emph{inversion number} $\inv(\sigma)$ of $\sigma$ is defined to be the number of inversions of $\sigma$. It is known \cite[Proposition~1.5.2]{BB} that 
\begin{equation} \label{eqn:inv-A}
\ell_S(\sigma) =\inv(\sigma).
\end{equation}
Recall that $(\mu,\nu)$ is a good (bad, respectively) edge in $\NC(W,c)$ if $\ell_S(\mu)<\ell_S(\nu)$ ($\ell_S(\mu)>\ell_S(\nu)$, respectively). By (\ref{eqn:n-cyc-A}) and (\ref{eqn:inv-A}), we have the following observations.

\medskip
\begin{lem} \label{lem:inv-good}
Let $(\mu,\nu)$ be an edge in $\NC(\mathfrak{S}_{n+1},c)$ with $\mu^{-1}\nu=(i,j)$. We have
\begin{enumerate}
\item $(\mu,\nu)$ is a good edge if $\inv(\mu)<\inv(\nu)$, and a bad edge if $\inv(\mu)>\inv(\nu)$.
\item The entries $i$ and $j$ occur in the same cycle of $\nu$, and in two different cycles of $\mu$.
\end{enumerate}
\end{lem}

\begin{exa} \label{exa:good-bad} {\rm
Let $\nu=2\,3\,4\,5\,6\,7\,8\,1$, written in cycle notation as $\nu=(1,2,\dots,8)$. Consider the transposition $t=(5,8)$. We obtain $\mu=\nu t=2\,3\,4\,5\,1\,7\,8\,6$, written in cycle notation as $\mu=(1,2,3,4,5)(6,7,8)$. Since $\inv(\nu)=7$ and $\inv(\mu)=6$, $(\mu,\nu)$ is a good edge. 
}
\end{exa}

The following observation will be used in Proposition \ref{pro:A-cycle-edges} to distinguish between good and bad edges in $\NC(\mathfrak{S}_{n+1},c)$.

\medskip
\begin{lem} \label{lem:inv(u)-inv(v)} If $\mu^{-1}\nu=(i, j)$ and $i<j$, we have $\inv(\mu)>\inv(\nu)$ if and only if $\nu(i)<\nu(j)$.
\end{lem}

\begin{proof} Note that $\mu(i)=\nu(j)$, $\mu(j)=\nu(i)$ and $\mu(k)=\nu(k)$ for $k\neq i, j$. The result follows from the definition of inversion number of a permutation.
\end{proof}

\medskip
\begin{pro} \label{pro:A-cycle-edges} Let $(\mu,\nu)$ be an edge of $\NC(\mathfrak{S}_{n+1},c)$ with $\mu^{-1}\nu=(i, j)$, where $i<j$. Suppose $(d_1, d_2, \dots, d_k)$ is the cycle of $\nu$ containing $i$ and $j$ for some integers $d_1<d_2<\cdots<d_k$. We have
\begin{enumerate}
\item $(\mu,\nu)$ is a bad edge if $i,j\in\{d_1,\dots,d_{k-1}\}$.
\item $(\mu,\nu)$ is a good edge if $i\in\{d_1,\dots,d_{k-1}\}$ and $j=d_k$.
\end{enumerate}
\end{pro}

\begin{proof} Note that $\nu(d_k)=d_1$ and $\nu(d_r)=d_{r+1}$ for each $r\in [k-1]$. By Lemmas \ref{lem:inv-good} and \ref{lem:inv(u)-inv(v)}, $(\mu,\nu)$ is a bad (good, respectively) edge if $\nu(i)<\nu(j)$ ($\nu(i)>\nu(j)$, respectively). In the former case, we have $i,j\in\{d_1,\dots,d_{k-1}\}$, and in the latter case $i\in\{d_1,\dots,d_{k-1}\}$ and $j=d_k$. 
\end{proof}

\medskip
\begin{exa} \label{exa:A-edges} {\rm
Suppose $\delta=(1,2,3,7,8)$ is a cycle of an element $\nu\in\NC(\mathfrak{S}_8,c)$ and $(\mu,\nu)$ is an edge with $\mu^{-1}\nu\leq\delta$.
Note that $(\mu,\nu)$ is a bad edge if $\mu^{-1}\nu\in\{(1,2), (2,3), (3,7), (1,3)$, $(2,7), (1,7)\}$, and a good edge if $\mu^{-1}\nu\in\{(1,8), (2,8), (3,8), (7,8)\}$.
}
\end{exa}

By Proposition \ref{pro:A-cycle-edges}, note that there are $\binom{k}{2}$ edges $(\mu,\nu)$ with transposition $\mu^{-1}\nu\leq\delta$, where $\delta=(d_1, d_2, \dots, d_k)$ is a cycle of $\nu$. Among them, $k-1$ edges are good. In the following result, these edges are partitioned into $k-1$ sets with exactly one good edge in each set.

\begin{lem} \label{lem:A-edge-classes} 
Suppose $\delta=(d_1,\dots,d_k)$ is a cycle of an element  $\nu\in\NC(\mathfrak{S}_{n+1},c)$ for some integers $ d_1<\cdots<d_k$, let $E_1,\dots,E_{k-1}$ be a partition of the edges $(\mu,\nu)$ with $\mu^{-1}\nu\leq\delta$, defined by
\begin{equation}
E_j:=\{(\mu,\nu) : \mu^{-1}\nu=(d_i, d_{i+j}), 1\le i\le k-j\}.
\end{equation}
Then the following properties hold.
\begin{enumerate}
\item $|E_j|=k-j$.
\item Each $E_j$ contains exactly one good edge.
\item For any $(\mu,\nu)\in E_j$, the cycle $\delta$ is split by the transposition $\mu^{-1}\nu$ into a $(k-j)$-cycle and a $j$-cycle.
\end{enumerate}
\end{lem}

\begin{proof}
(i) That $|E_j|=k-j$ is clear. (ii) By Proposition \ref{pro:A-cycle-edges}, the edge $(\mu,\nu)$ with $\mu^{-1}\nu=(d_{k-j},d_k)$ is the only good edge in $E_j$. (iii) For any $(\mu,\nu)\in E_j$, say $\mu^{-1}\nu=(d_i,d_{i+j})$, we observe that $(d_1,\dots,d_i,d_{i+j+1},\dots,d_k)$ and $(d_{i+1},\dots,d_{i+j})$ are two cycles of $\mu$.
\end{proof}

\medskip
\subsection{The bijection $\varphi_A$}
In what follows we shall establish a bijection $\varphi_A$ between the parking functions in $\PF^A$ and the maximal chains in $\NC(\mathfrak{S}_{n+1},c)$.

Let $\alpha=(a_1,\dots,a_n)\in\PF^A_n$ with outcome $\pi=(p_1,\dots,p_n)$. The following result shows that every entry $p_i$ determines a decomposition of the prefix $a_1,\dots, a_{i-1}$ of $\alpha$ into two parking functions of certain sizes.

\begin{lem} \label{lem:pn=k} For all $j<i$, we have
\begin{enumerate}
\item  $p_j<p_i$ if and only if $a_j<p_i$;
\item  $p_j>p_i$ if and only if $a_j>p_i$.
\end{enumerate}
\end{lem}

\begin{proof} Let $p_i=\ell$. Since the $j$th car precedes the $i$th car, the parking space $\ell$ is empty when the $j$th car has parked.  Thus, $a_j<\ell$ implies $p_j<\ell$. Since $a_j\le p_j$,  it follows that $p_j<\ell$ implies $a_j<\ell$. Thus, the assertion (i) follows.

The assertion (ii) is proved by the same argument.
\end{proof}

\medskip
Let $\varphi_A(\alpha)=\ww$, where $\ww: e=\omega_0<\omega_1<\cdots<\omega_n=c$, be the corresponding maximal chain of $\alpha$, and let $t_1t_2\cdots t_n$ be the corresponding factorization of $c=(1,2,\dots,n+1)$, where $t_i=\omega^{-1}_{i-1}\omega_i$. We shall construct $t_1,t_2,\dots, t_n$ in reverse order by the following procedure. When working on $t_i$, the cycles of $\omega_i$ have been ordered $\delta_1\cdots \delta_{n-i+1}$, and each cycle $\delta_r$ is associated with a tuple $\zeta_r$ of consecutive integers of the same size as $\delta_r$, where $\zeta_1,\dots,\zeta_{n-i+1}$ form an interval partition of $[n+1]$.   
In terms of parking spaces, this interval partition is obtained by cutting after every empty space, i.e., the non-maximal entries of each interval $\zeta_r$ have been occupied.
Initially, we associate $\omega_n=\delta_1$ with the $(n+1)$-tuple $\zeta_1=(1,2,\dots,n+1)$.

\medskip
\noindent
{\bf Algorithm A}

For $i=n,\dots,1$, the reflection $t_i$ is determined by $a_i$ and $p_i$ as follows.

(i) Among the tuples $\zeta_1,\dots,\zeta_{n-i+1}$, find the tuple that contains the integer $p_i$, say $\zeta_r$. 
Let $\zeta_r=(h+1,h+2,\dots,h+m+1)$ for some integers $h$ and $m$, say $a_i=h+j$ and $p_i=h+k$, where $j\le k$. 
Suppose the cycle $\delta_r$ is $(d_1,d_2,\dots,d_{m+1})$ for some integers $d_1<d_2<\cdots<d_{m+1}$.

(ii) Set $t_i=(d_{j},d_{m+1-k+j})$.
Note that $\delta_r$ is split by $t_i$ into a $k$-cycle $\delta_{r,1}$ and a $(m-k+1)$-cycle $\delta_{r,2}$, written in the order $\delta_{r,1}\delta_{r,2}$, where
\begin{equation}
\delta_{r,1}=(d_1,\dots,d_{j},d_{m+2-k+j},\dots,d_{m+1}),\qquad \delta_{r,2}=(d_{j+1},\dots,d_{m+1-k+j}).
\end{equation}
Their associated tuples $\zeta_{r,1}, \zeta_{r,2}$ are obtained by cutting $\zeta_r$ at the immediate right of $p_i$, i.e.,  
\begin{equation} \label{eqn:zeta-sequences}
\zeta_{r,1}=(h+1,\dots,h+k),\qquad \zeta_{r,2}=(h+k+1,\dots,h+m+1). 
\end{equation}

\medskip
\begin{exa} \label{exa:PF-to-noncrossing-A} {\rm
For $n=7$, let $\alpha=(2,2,6,1,6,1,3)$ with outcome $\pi=(2,3,6,1,7,4,5)$. Let $\varphi_A(\alpha)=t_1t_2\cdots t_7$, where $t_i=\omega^{-1}_{i-1}\omega_i$. The construction of $t_1,t_2,\cdots,t_7$ is shown in Table \ref{tab:construct-type-A-noncrossing}, where the good and bad edges are indicated by $\circ$ and $\times$, respectively, in the last column. Some initial steps are described below.
\begin{itemize}
\item $a_7=3$ and $p_7=5$. By $\delta_1=(1,\dots,8)$, we set $t_7=(3,6)$ and then the cycle $\delta_1$ is split into $\delta_{1,1}=(1,2,3,7,8)$ and $\delta_{1,2}=(4,5,6)$ associated with the tuples $\zeta_{1,1}=(1,2,3,4,5)$ and $\zeta_{1,2}=(6,7,8)$, respectively.
\item $a_6=1$ and $p_6=4$. Note that $p_6$ is contained in the tuple $\zeta_1=(1,2,3,4,5)$ associated to $\delta_1=(d_1,\dots,d_5)=(1,2,3,7,8)$. We set $t_6=(d_1,d_2)=(1,2)$ and then the cycle $\delta_1$ is split into
$\delta_{1,1}=(1,3,7,8)$ and $\delta_{1,2}=(2)$ associated with the tuples $\zeta_{1,1}=(1,2,3,4)$ and $\zeta_{1,2}=(5)$, respectively. 
\end{itemize}
}
\end{exa}

\begin{table}[ht]
\caption{The construction of the maximal chain $\varphi_A(\alpha)=t_1t_2\cdots t_7$ in Example \ref{exa:PF-to-noncrossing-A}.}
\centering
\begin{tabular}{c|ccccllc|c}
$i$ &  &  $a_i$ & $p_i$  & $t_i$ & \multicolumn{1}{c}{$\omega_i$} & \multicolumn{1}{c}{$\zeta_1\cdots\zeta_{n-i+1}$} &   & $\circ\times$  \\[2pt]
\hline
  7 &  &  3  &  5   &           &  $(1,2,3,4,5,6,7,8)$ &  $(1,2,3,4,5,6,7,8)$     &       &      \\[-6pt]
    &  &     &      &  $(3,6)$   &    &       &       &   $\times$ \\  
  6 &  &  1  &  4   &            &  $(1,2,3,7,8)(4,5,6)$   &  $(1,2,3,4,5)(6,7,8)$      &    &    \\[-6pt]
    &  &     &      &  $(1,2)$   &     &      &    &   $\times$ \\
  5 &  &  6  &  7   &            &  $(1,3,7,8)(2)(4,5,6)$  &  $(1,2,3,4)(5)(6,7,8)$     &    &    \\[-6pt]
    &  &     &      &  $(4,5)$   &     &      &    &   $\times$ \\
  4 &  &  1  &  1   &            &  $(1,3,7,8)(2)(4,6)(5)$ &  $(1,2,3,4)(5)(6,7)(8)$    &    &   \\[-6pt]
    &  &     &      &  $(1,8)$   &    &       &    &   $\circ$ \\ 
  3 &  &  6  & 6    &     &  $(1)(3,7,8)(2)(4,6)(5)$ &  $(1)(2,3,4)(5)(6,7)(8)$   &        &    \\[-6pt]
    &  &     &      &  $(4,6)$   &    &      &        &   $\circ$ \\
  2 &  &  2  & 3    &       &  $(1)(3,7,8)(2)(4)(6)(5)$  &   $(1)(2,3,4)(5)(6)(7)(8)$  &     &    \\[-6pt]
    &  &     &      &  $(3,7)$   &     &      &     &   $\times$ \\ 
  1 &  &  2  & 2    &      &  $(1)(3,8)(7)(2)(4)(6)(5)$  &  $(1)(2,3)(4)(5)(6)(7)(8)$ &       &     \\[-6pt]
    &  &     &      &  $(3,8)$   &     &    &       &   $\circ$ \\
  0 &  &     &          &   &   $(1)(3)(8)(7)(2)(4)(6)(5)$ &  $(1)(2)(3)(4)(5)(6)(7)(8)$                  &    &    
\end{tabular}
\label{tab:construct-type-A-noncrossing}
\end{table}

\medskip
\begin{pro} \label{pro:preserving-A} We have
\begin{enumerate}
\item  The maximal chain $\varphi_A(\alpha)$ is well-defined.
\item  For $1\le i\le n$, $(\omega_{i-1},\omega_i)$ is a bad edge in $\varphi_A(\alpha)$ if and only if the $i$th car of $\alpha$ is reluctant.
\end{enumerate}
\end{pro}

\begin{proof} (i) By the determination of $t_i$ in Algorithm A, it suffices to prove that $a_i$ occurs in the same tuple as $p_i$ for each $i$. In the initial step, the entries $a_n, p_n$ are in the $(n+1)$-tuple $\zeta_1=(1,2,\dots,n+1)$.

For $i\le n$, suppose the entries $a_i,p_i$ occur in the  tuple $\zeta_r=(h+1,\dots,h+m+1)$ associated to a cycle $\delta_r=(d_1,\dots,d_{m+1})$ of $\omega_i$, say $a_i=h+j$ and $p_i=h+k$.
By (\ref{eqn:zeta-sequences}), the tuples $\zeta_{r,1}$ and $\zeta_{r,2}$ are obtained from $\zeta_r$ by cutting after the entry $p_i$. By Lemma \ref{lem:pn=k}, $a_{i-1}$ occurs in the same tuple as $p_{i-1}$.  Hence $\varphi_A(\alpha)$ is well-defined.

(ii) We set $t_i=(d_{j},d_{m+1-k+j})$, where $t_i=\omega_{i-1}^{-1}\omega_i$. By Proposition \ref{pro:A-cycle-edges}, $(\omega_{i-1},\omega_i)$ is a good (bad, respectively) edge if $j=k$ ($j<k$, respectively), in which case $a_i=p_i$ ($a_i<p_i$, respectively). The result follows.
\end{proof}

\medskip
The map $\varphi_A^{-1}$ can be established by a reverse operation.  Given a maximal chain $\ww:e=\omega_0<\omega_1<\cdots<\omega_n=c$ in $\NC(\mathfrak{S}_{n+1},c)$, let $t_1t_2\cdots t_n$ be the factorization of $c$, where $t_i=\omega^{-1}_{i-1}\omega_i$. We shall construct the corresponding parking function $\varphi_A^{-1}(\ww)=(a_1,\dots,a_n)$ with outcome $\pi=(p_1,\dots,p_n)$ in reverse order by the following procedure. Likewise, when working on $a_i$ and $p_i$, the cycles of $\omega_i$ have been ordered $\delta_1\cdots \delta_{n-i+1}$ with an interval partition $\zeta_1,\dots,\zeta_{n-i+1}$ of $[n+1]$. Initially, we associate $\omega_n=\delta_1$ with the $(n+1)$-tuple $\zeta_1=(1,2,\dots,n+1)$. 

\medskip
\noindent
{\bf Algorithm B}

For $i=n,\dots,1$, the entries $a_i$ and $p_i$ are determined by $t_i$ as follows.

(i) Among the cycles $\delta_1, \dots, \delta_{n-i+1}$ of $\omega_i$, find the cycle that is comparable with $t_i$, say $t_i\leq\delta_r$. Let $\delta_r=(d_1,\dots,d_{m+1})$ for some integers $d_1<\cdots<d_{m+1}$ and let $\zeta_r=(h+1,\dots,h+m+1)$ for some integer $h$. Suppose $t_i=(d_{j},d_{k})$, where $j<k$.

(ii) Set $a_i=h+j$ and $p_i=h+m+1-k+j$. 
Note that the cycle $\delta_r$ is split by $t_i$ into a $(m+1-k+j)$-cycle $\delta_{r,1}$ and a $(k-j)$-cycle $\delta_{r,2}$, where
\begin{equation*}
\delta_{r,1}=(d_1,\dots,d_{j},d_{k+1},\dots,d_{m+1}); \qquad \delta_{r,2}=(d_{j+1},\dots,d_{k}).
\end{equation*}
For their associated tuples, we cut $\zeta_r$ after the entry $p_i$ into 
\begin{equation*}
\zeta_{r,1}=(h+1,\dots,h+m+1-k+j);\qquad \zeta_{r,2}=(h+m+2-k+j,\dots, h+m+1).
\end{equation*}

\medskip
The proof of Theorem \ref{thm:type-A-PF-NC} is completed. The correspondence of the bijection $\varphi_A$ for $n=3$ is listed in Table \ref{tab:map-PF-A3} for reference.

\begin{table}[ht]
\caption{The map $\varphi_A:\alpha\mapsto\ww=t_1t_2t_3$ for all $\alpha\in\PF^A_3$.}
\centering
{\small
\begin{tabular}{cccc|ccc|c}
\hline
 $\alpha$  & & $\pi$  & & & $\ww$ & &  $\circ\times$\\
\hline
  123   & &  123   & & & $(1,2)(2,3)(3,4)$    & &  \multirow{6}{0.8cm}{$\circ\circ\circ$}    \\[2pt]
  132   & &  132   & & & $(1,2)(3,4)(2,4)$    & &      \\[2pt]
  213   & &  213   & & & $(2,3)(1,3)(3,4)$    & &      \\[2pt]
  231   & &  231   & & & $(2,3)(3,4)(1,4)$    & &      \\[2pt]
  312   & &  312   & & & $(3,4)(1,2)(2,4)$    & &      \\[2pt]
  321   & &  321   & & & $(3,4)(2,4)(1,4)$    & &      \\
\hline  
  121   & &  123   & & & $(1,3)(3,4)(1,2)$    & &  \multirow{6}{0.8cm}{$\circ\circ\times$}    \\[2pt]
  122   & &  123   & & & $(1,2)(2,4)(2,3)$    & &      \\[2pt]
  131   & &  132   & & & $(1,4)(2,3)(1,3)$    & &      \\[2pt]
  211   & &  213   & & & $(3,4)(1,4)(1,2)$    & &      \\[2pt]
  212   & &  213   & & & $(2,4)(1,4)(2,3)$    & &      \\[2pt]
  311   & &  312   & & & $(2,3)(1,4)(1,3)$    & &      \\
\hline
  113   & &  123   & & & $(1,3)(1,2)(3,4)$    & &  \multirow{2}{0.8cm}{$\circ\times\circ$}    \\[2pt]
  221   & &  231   & & & $(2,4)(2,3)(1,4)$    & &      \\
\hline
  111   & &  123   & & & $(1,4)(1,3)(1,2)$    & &  \multirow{2}{0.8cm}{$\circ\times\times$}    \\[2pt]
  112   & &  123   & & & $(1,4)(1,2)(2,3)$    & &      \\
\hline
\end{tabular}
}
\label{tab:map-PF-A3}
\end{table}

\medskip
Let $A_n(q):=M(\mathfrak{S}_{n+1},q)$. We observe that the polynomial $A_n(q)$ satisfies the following recurrence relation.

\smallskip
\begin{pro} 
For all $n\ge 1$, we have
\begin{equation*}
A_n(q)=\sum_{k=1}^{n} \big(1+(n-k)q\big) \binom{n-1}{k-1} A_{k-1}(q)A_{n-k}(q),
\end{equation*}
with  $A_0(q)=1$.
\end{pro}

\begin{proof}  We enumerate the factorization $t_1t_2\cdots t_n$ of $c$ into transpositions inductively, which corresponds to a maximal chain in $\NC(\mathfrak{S}_{n+1},c)$.

Consider a family $T_1,\dots,T_n$ of sets of transpositions given by
\begin{equation}
T_k:=\{(j, j+k) : 1\le j\le n+1-k\}.
\end{equation}
The transposition $t_n$ is in one of the sets $T_1,\dots,T_n$, say $t_n\in T_k$. There are $n+1-k$ choices for $t_n$ (i.e., $|T_k|=n+1-k$).
By Lemma \ref{lem:A-edge-classes}, we observe that $(\omega_{n-1},\omega_n)$ is a bad edge if $t_n=(j,j+k)$ for each $j\in [n-k]$, and a good edge if $t_n=(n+1-k, n+1)$. Moreover, $\omega_n$ is split by $t_n$ into an $(n+1-k)$-cycle $\delta_1$ and a $k$-cycle $\delta_2$, and among  $t_1,t_2,\dots,t_{n-1}$ there are $n-k$ ($k-1$, respectively) transpositions which are comparable with $\delta_1$ ($\delta_2$, respectively).
The generating $q$-polynomial for the maximal chains in the $\delta_1$-noncrossing partition lattice ($\delta_2$-noncrossing partition lattice, respectively) is $A_{n-k}(q)$ ($A_{k-1}(q)$, respectively). The result follows.
\end{proof}

\section{Noncrossing partitions and parking functions of type $B$}
\subsection{Noncrossing partitions of type $B$}
A set partition $\PP$ of $[\pm n]$ is \emph{invariant} under sign change if $P$ is a block of $\PP$ then $-P$ is also a block of $\PP$, where $-P=\{-x: x\in P\}$. We identify $\PP$ with a partition of $[2n]$ by sending $i\mapsto i$ for $i\in\{1,2,\dots,n\}$ and $i\mapsto n-i$ for $i\in\{-1,-2,\dots,-n\}$.
As introduced by Reiner \cite{Reiner}, the lattice $\NC^B_n$ of noncrossing partitions of type $B$ is defined to be the sublattice of $\NC_{2n}$ consisting of partitions which are invariant under sign change. An invariant block $P=-P$ of a partition is called \emph{zero block}. By the noncrossing condition, there is at most one zero block in $\PP$ and the nonzero blocks come in pairs.

Consider the group $B_n$ of signed permutations. Every permutation $\omega\in B_n$ maps to the partition of $[\pm n]$ given by the cycle structure of $\omega$. With respect to the Coxeter element 
\begin{equation*}
c=(1, 2, \dots, n, -1, -2, \dots, -n),
\end{equation*}
this map defines a poset isomorphism from $\NC(B_n,c)$ to $\NC^B_n$ (see \cite{Bessis}, \cite{BW-02}, \cite{BW}).
In one-line notation, we write $\omega=\omega(1)\omega(2)\cdots\omega(n)$ with bars to indicate negative numbers. We also write $\omega$ in disjoint cycles. For example, if $\omega$ is determined by $\omega(1)=-3$, $\omega(2)=-2$, $\omega(3)=5$, $\omega(4)=4$, $\omega(5)=6$ and $\omega(6)=-1$, we write
\begin{equation*}
\omega= \overline{3}\,\overline{2}\,5\,4\,6\,\overline{1}=(1,-3,-5,-6)(-1,3,5,6)(2,-2)(4)(-4).
\end{equation*}
Note that $\omega$ maps to $\{\{1,-3,-5,-6\}$, $\{-1,3,5,6\}$, $\{2,-2\}$, $\{4\},\{-4\}\}$, which is a noncrossing partition of type $B$ with zero block $\{2,-2\}$. 

The reflections of $B_n$ are signed permutations of the following three kinds \cite[Proposition~8.1.5]{BB}:
\begin{equation} \label{eqn:B-reflections}
\begin{aligned}
(i,-i) &\quad\mbox{for all $1\le i\le n$,} \\
(i,j)(-i,-j) &\quad \mbox{for all $1\le i<j\le n$,}\\
(i,-j)(-i,j) &\quad \mbox{for all $1\le i<j\le n$.}
\end{aligned}
\end{equation}
A maximal chain $\ww: e=\omega_0<\omega_1<\cdots<\omega_n=c$ in $\NC(B_n,c)$ corresponds to a factorization $t_1t_2\cdots t_n$ of $c$ into $n$ reflections, where $t_i=\omega^{-1}_{i-1}\omega_i$. 

For any $\omega\in\NC(B_n,c)$, the cycle structure of $\omega$ consists of at most one cycle of the form $\delta_0=(d_1,\dots,d_k,-d_1,\dots,-d_k)$ and some pairs of the form
\begin{equation} \label{eqn:nonzero}
\delta_1=(d_1,\dots,d_{\ell},-d_{\ell+1},\dots,-d_k)(-d_1,\dots,-d_{\ell},d_{\ell+1},\dots,d_k),
\end{equation}
where $0<d_1<\cdots<d_k$ and $1\le \ell\le k$.
Following \cite{BW-02}, we call $\delta_0$ ($\delta_1$, respectively) a \emph{balanced cycle} (\emph{paired cycle}, respectively). Note that a reflection $t$ satisfies $t\leq\delta_0$ if $t=(i,-i)$ or $(i,j)(-i,-j)$ for some $i,j\in\{\pm d_1,\dots,\pm d_k\}$, and $t\leq\delta_1$ if $t=(i,j)(-i,-j)$ for some $i,j\in \{d_1,\dots,d_{\ell}, -d_{\ell+1},\dots,-d_k\}$.  
Let $\npc(\omega)$ denote the number of paired cycles of $\omega$. Note that 
\begin{equation} \label{eqn:n-nzp}
\ell_T(\omega)=n-\npc(\omega).
\end{equation}

We recall some facts about signed permutations; see \cite[Section~8.1]{BB} for details. The inversion number of $\omega$, denoted $\inv_B(\omega)$, can be calculated by
\begin{equation}
\inv_B(\omega):=\inv(\omega(1)\cdots\omega(n))+\sum_{\{j\in [n]\,:\,\omega(j)<0\}} |\omega(j)|.
\end{equation}
For example, if $\omega=\overline{3}\,\overline{2}\,5\,4\,6\,\overline{1}$, then $\inv_B(\omega)=4+6=10$. It is known \cite[Proposition~8.1.1]{BB} that
\begin{equation}
\ell_S(\omega)=\inv_B(\omega).
\end{equation}

The following results come from \cite[Proposition~8.1.6]{BB}, which will be used in Propositions \ref{pro:zero-cycle} and \ref{pro:nonzero-cycle} to distinguish between good and bad edges in $\NC(B_n,c)$. 

\medskip
\begin{lem} \label{lem:inv_B(u)-inv_B(v)}
For any edge $(\mu,\nu)$ of $\NC(B_n,c)$, the following results hold.
\begin{enumerate}
\item If $\mu^{-1}\nu=(i,-i)$, where $1\le i\le n$, we have $\inv_B(\mu)>\inv_B(\nu)$ if and only if $\nu(i)>0$.
\item If $\mu^{-1}\nu=(i,j)(-i,-j)$, where $1\le i<j\le n$, we have $\inv_B(\mu)>\inv_B(\nu)$ if and only if $\nu(i)<\nu(j)$ or,  equivalently, $\nu(-j)<\nu(-i)$.
\item If $\mu^{-1}\nu=(i,-j)(-i,j)$, where $1\le i<j\le n$, we have $\inv_B(\mu)>\inv_B(\nu)$ if and only if $\nu(-i)<\nu(j)$ or,  equivalently, $\nu(-j)<\nu(i)$.
\end{enumerate}
\end{lem}

The following results characterize the good and bad edges whose reflections are comparable with a balanced cycle.

\medskip
\begin{pro} \label{pro:zero-cycle} Suppose $(\mu,\nu)$ is an edge in $\NC(B_n,c)$ with  reflection $t=\mu^{-1}\nu\leq\delta$, where $\delta=(d_1,\dots,d_m,-d_1,\dots,-d_m)$ is the balanced cycle of $\nu$, for some integers  $0<d_1<\cdots<d_m$. Then the following results hold.
\begin{enumerate}
\item For $t=(i,-i)$, where $i>0$, $(\mu,\nu)$ is a bad edge if $i\in\{d_1,\dots,d_{m-1}\}$, and a good edge if $i=d_m$.
\item For $t=(i,j)(-i,-j)$, where $0<i<j$,   $(\mu,\nu)$ is a bad edge if $i,j\in\{d_1,\dots,d_{m-1}\}$, and a good edge if $i\in\{d_1,\dots,d_{m-1}\}$ and $j=d_m$.
\item For $t=(i,-j)(-i,j)$, where $0<i<j$, $(\mu,\nu)$ is a bad edge for all $i,j\in\{d_1,\dots,d_m\}$.
\end{enumerate}
\end{pro}

\begin{proof} (i) Note that $\nu(d_m)=-d_1$ and $\nu(d_r)=d_{r+1}$ for each $r\in [m-1]$.
For $t=(i,-i)$ and $i>0$, if $\nu(i)>0$ then by Lemma \ref{lem:inv_B(u)-inv_B(v)}(i) we have $\inv_B(\mu)>\inv_B(\nu)$ and hence $(\mu,\nu)$ is a bad edge. In this case, $i\in\{d_1,\dots,d_{m-1}\}$. Otherwise,  $\nu(i)<0$. In this case, $i=d_m$ and $(\mu,\nu)$ is a good edge.

(ii) For $t=(i,j)(-i,-j)$ and $0<i<j$, if $\nu(i)<\nu(j)$ then by Lemma \ref{lem:inv_B(u)-inv_B(v)}(ii) we have $\inv_B(\mu)>\inv_B(\nu)$ and hence $(\mu,\nu)$ is a bad edge. In this case, $i,j\in\{d_1,\dots,d_{m-1}\}$. Otherwise, $\nu(i)>\nu(j)$. In this case, $i\in\{d_1,\dots,d_{m-1}\}$, $j=d_m$, and $(\mu,\nu)$ is a good edge.

(iii) For $t=(i,-j)(-i,j)$ and $0<i<j$, note that $\nu(-i)<\nu(j)$ for all $i,j\in\{d_1,\dots,d_m\}$ since $\nu(d_m)=-d_1$ and $\nu(-d_r)=-d_{r+1}$ for each $r\in [m-1]$. By Lemma \ref{lem:inv_B(u)-inv_B(v)}(iii), $\inv_B(\mu)>\inv_B(\nu)$ and hence $(\mu,\nu)$ is a bad edge.
\end{proof}

\medskip
By Proposition \ref{pro:zero-cycle}, there are $m^2$ edges $(\mu,\nu)$ with reflection $\mu^{-1}\nu\leq\delta$, where $\delta=(d_1,\dots,d_m$, $-d_1,\dots,-d_m)$ is the balanced cycle  of $\nu$. Among them, $m$ edges are good. For convenience, we write $\delta=(d_1,\dots,d_{2m})$ and use the notation $d_{m+i}$ ($d_{2m+i}$, respectively) to stand for $-d_i$ ($d_i$, respectively) for all $i\in [m]$ when working with congruence of indices modulo $2m$. 

\begin{lem} \label{lem:zero-cycle-edge-classes} 
Suppose $\delta=(d_1,\dots,d_{2m})$ is the balanced cycle of an element $\nu\in \NC(B_n,c)$,  where $0<d_1<\cdots<d_{m}$ and $d_{m+i}=-d_i$ for all $i\in [m]$. Let $E_1,\dots,E_m$ be a partition of the edges $(\mu,\nu)$ with $\mu^{-1}\nu\leq\delta$, defined by
\begin{equation*}
E_j := \begin{cases}
\{(\mu,\nu) : \mu^{-1}\nu=(d_i,d_{j+i})(d_{m+i},d_{m+j+i}), 1\le i\le m\} &\text{if $1\le j\le m-1$;} \\
\{(\mu,\nu) : \mu^{-1}\nu=(d_i,d_{m+i}), 1\le i\le m\} &\text{if $j=m$.}
\end{cases}
\end{equation*}
Then for $1\le j\le m$ the following properties hold.
\begin{enumerate}
\item $|E_j|=m$.
\item Each $E_j$ contains exactly one good edge.
\item For any $(\mu,\nu)\in E_j$, the reflection $\mu^{-1}\nu$ splits the cycle $\delta$ into a balanced cycle of size $2(m-j)$ and a paired cycle of size $j$.
\end{enumerate}
\end{lem}

\begin{proof}
(i) That $|E_j|=m$ is clear. (ii) By Proposition \ref{pro:zero-cycle}, the edge $(\mu,\nu)$ with $\mu^{-1}\nu=(d_m,d_{2m})$ is the only good edge in $E_m$, and the edge with $\mu^{-1}\nu=(d_{m-j},d_m)(d_{2m-j},d_{2m})$ is the only good edge in $E_j$ for $1\le j\le m-1$. 

(iii) For any edge $(\mu,\nu)\in E_m$, say $\mu^{-1}\nu=(d_i,d_{m+i})$, the cycle $\delta$ is split by $\mu^{-1}\nu$ into a paired cycle $(d_1,\dots,d_i,d_{m+i+1},\dots,d_{2m})(d_{m+1},\dots,d_{m+i}$, $d_{i+1},\dots,d_m)$ of size $m$.
For $1\le j\le m-1$ and any edge $(\mu,\nu)\in E_j$, say $\mu^{-1}\nu=(d_i,d_{j+i})(d_{m+i},d_{m+j+i})$, the cycle $\delta$ is split by $\mu^{-1}\nu$ into a balanced cycle $(d_1,\dots,d_i,d_{i+j+1},\dots,d_{m+i},d_{m+i+j+1},\dots,d_{2m})$ and a paired cycle $(d_{i+1},\dots,d_{i+j})(d_{m+i+1},\dots,d_{m+i+j})$. The result follows.
\end{proof}

\medskip
\begin{exa} \label{exa:zero-edges} {\rm
Suppose $\delta=(1,2,7,8,-1,-2,-7,-8)$ is the balanced cycle of $\nu$ and $(\mu,\nu)$ is an edge in $\NC(B_8,c)$ with $\mu^{-1}\nu\leq\delta$. 
Note that $(\mu,\nu)$ is a good edge if $\mu^{-1}\nu\in\{(8,-8)$, $(1,8)(-1,-8)$, $(2,8)(-2,-8), (7,8)(-7,-8)\}$. Moreover, the reflections comparable with $\delta$ can be partitioned into the following sets such that each set contains one reflection corresponding to a good edge, i.e., 

\begin{align*}
 &\{(1,2)(-1,-2),(2,7)(-2,-7),(7,8)(-7,-8),(1,-8)(-1,8)\}; \\
 &\{(1,7)(-1,-7),(2,8)(-2,-8),(1,-7)(-1,7),(2,-8)(-2,8)\}; \\
 &\{(1,8)(-1,-8),(1,-2)(-1,2), (2,-7)(-2,7), (7,-8)(-7,8)\}; \\
 &\{(1,-1),(2,-2),(7,-7),(8,-8)\}.\\
\end{align*}
}
\end{exa}

\medskip
The following results characterize the good and bad edges in $\NC(B_n,c)$ whose reflections are comparable with a paired cycle.

\medskip
\begin{pro} \label{pro:nonzero-cycle} Suppose $(\mu,\nu)$ is an edge in $\NC(B_n,c)$ with  reflection $t=\mu^{-1}\nu\leq\delta$, where 
$\delta=(d_1,\dots,d_{\ell},-d_{\ell+1},\dots,-d_k)(-d_1,\dots,-d_{\ell},d_{\ell+1},\dots,d_k)$
is a paired cycle of $\nu$, for some integers $0<d_1<\cdots<d_k$ and $1\le\ell\le k$. Write $\delta=(e_1,\dots,e_k)(-e_1,\dots,-e_k)$, where 
\begin{equation*}
e_r=\begin{cases}
d_r &\mbox{if $1\le r\le \ell$;} \\
-d_r &\mbox{if $\ell+1\le r\le k$.}
\end{cases}
\end{equation*}
Then $t=(e_i,e_j)(-e_i,-e_j)$ for some $1\le i<j\le k$. Moreover, $(\mu,\nu)$ is a good edge if $i$ and $j$ satisfy one of the following conditions; otherwise, $(\mu,\nu)$ is a bad edge:
\begin{enumerate}
\item $1\le i<\ell$ and $j=\ell$;
\item $\ell\le i<k$ and $j=k$.
\end{enumerate}
\end{pro}

\begin{proof} If $\ell=k$ then $\delta=(d_1,\dots,d_{\ell})(-d_1,\dots,-d_{\ell})$, where $0<d_1<\cdots<d_{\ell}$.  By the same argument as in the proof of Proposition \ref{pro:zero-cycle}(ii), $(\mu,\nu)$ is a bad edge if $e_i,e_j\in\{d_1,\dots,d_{\ell-1}\}$, and a good edge if $e_i\in\{d_1,\dots,d_{\ell-1}\}$ and $e_j=d_{\ell}$.

Let $\ell<k$.
Note that $\nu(d_{\ell})=-d_{\ell+1}$, $\nu(d_k)=-d_1$, and $\nu(d_r)=d_{r+1}$ for each $r\neq \ell,k$. 

Case 1. If $e_ie_j>0$ then either $e_i,e_j\in\{d_1,\dots,d_{\ell}\}$ or $e_i,e_j\in\{-d_{\ell+1},\dots,-d_k\}$. Note that if $e_i\in\{d_1,\dots,d_{\ell-1}\}$ and $e_j=d_{\ell}$ then $\nu(e_i)>\nu(e_j)$ and that if $e_i\in\{-d_{\ell+1},\dots,-d_{k-1}\}$ and $e_j=-d_k$ then $\nu(e_j)>\nu(e_i)$. By Lemma \ref{lem:inv_B(u)-inv_B(v)}(ii), $\inv_B(\mu)<\inv_B(\nu)$ and hence $(\mu,\nu)$ is a good edge. Otherwise, $e_i,e_j\in\{d_1,\dots,d_{\ell-1}\}$ ($e_i,e_j\in\{-d_{\ell+1},\dots,-d_{k-1}\}$, respectively). In this case, $\nu(e_i)<\nu(e_j)$ ($\nu(e_j)<\nu(e_i)$, respectively) and hence $(\mu,\nu)$ is a bad edge.

Case 2. If $e_ie_j<0$, there are four cases:
\begin{itemize}
\item $e_i\in\{d_1,\dots,d_{\ell-1}\}$ and $e_j\in\{-d_{\ell+1},\dots,-d_{k-1}\}$. We have $\nu(e_i)=\nu(d_i)=d_{i+1}$ and $\nu(e_j)=\nu(-d_j)=-d_{j+1}$. Thus, $\nu(e_j)<\nu(e_i)$. By Lemma \ref{lem:inv_B(u)-inv_B(v)}(iii), $\inv_B(\mu)>\inv_B(\nu)$ and hence $(\mu,\nu)$ is a bad edge.
\item $e_i=d_{\ell}$ and $e_j\in\{-d_{\ell+1},\dots,-d_{k-1}\}$. We have $\nu(e_i)=\nu(d_{\ell})=-d_{\ell+1}$ and $-d_{\ell+2}\ge\nu(e_j)\ge -d_k$. Thus, $\nu(e_j)<\nu(e_i)$. By Lemma \ref{lem:inv_B(u)-inv_B(v)}(iii), $\inv_B(\mu)>\inv_B(\nu)$ and hence $(\mu,\nu)$ is a bad edge.
\item $e_i\in\{d_1,\dots,d_{\ell-1}\}$ and $e_j=-d_k$. We have $d_2\le \nu(e_i)\le d_{\ell}$ and $\nu(e_j)=\nu(-d_k)=d_1$. Thus, $\nu(e_j)<\nu(e_i)$. By Lemma \ref{lem:inv_B(u)-inv_B(v)}(iii), $\inv_B(\mu)>\inv_B(\nu)$ and hence $(\mu,\nu)$ is a bad edge.
\item $e_i=d_{\ell}$ and $e_j=-d_k$. Since $\nu(e_j)=\nu(-d_k)=d_1>-d_{\ell+1}=\nu(d_{\ell})=\nu(e_i)$, by Lemma \ref{lem:inv_B(u)-inv_B(v)}(iii), $\inv_B(\mu)<\inv_B(\nu)$ and hence $(\mu,\nu)$ is a good edge.
\end{itemize}
The result follows.
\end{proof}

\begin{exa} \label{exa:nonzero-edges} {\rm
Suppose $\delta=(2,3,-7,-8,-9)(-2,-3,7,8,9)$  is a paired cycle of $\nu$ and $(\mu,\nu)$ is an edge in $\NC(B_9,c)$ with $\mu^{-1}\nu\leq\delta$. 
Note that $(\mu,\nu)$ is a good edge if $\mu^{-1}\nu$ is one of the reflections in the set
\begin{equation*}
\{(2,3)(-2,-3), (3,-9)(-3,9), (-7,-9)(7,9), (-8,-9)(8,9)\},
\end{equation*}
and a bad edge if $\mu^{-1}\nu$ is one of the reflections in the set
\begin{equation*}
\{(2,-7)(-2,7),  (2,-8)(-2,8), (2,-9)(-2,9), (3,-7)(-3,7), (3,-8)(-3,8), (-7,-8)(7,8)\}.
\end{equation*}
}
\end{exa}

By Proposition \ref{pro:nonzero-cycle}, there are $\binom{k}{2}$ edges $(\mu,\nu)$ with reflection $\mu^{-1}\nu\leq\delta$, where $\delta=(d_1,\dots,d_{\ell}$, $-d_{\ell+1},\dots,-d_k)(-d_1,\dots,-d_{\ell},d_{\ell+1},\dots,d_k)$ is a paired cycle of $\nu$. In the following result, these edges are partitioned into $k-1$ sets with exactly one good edge in each set.

\medskip
\begin{lem} \label{lem:nonzero-cycle-edge-classes} Suppose $(\mu,\nu)$ is an edge in $\NC(B_n,c)$ with  reflection $t=\mu^{-1}\nu\leq\delta$, where 
$\delta=(d_1,\dots,d_{\ell},-d_{\ell+1},\dots,-d_k)(-d_1,\dots,-d_{\ell},d_{\ell+1},\dots,d_k)$
is a paired cycle of $\nu$, for some integers $0<d_1<\cdots<d_k$ and $1\le\ell\le k$. Write $\delta=(f_1,\dots,f_k)(-f_1,\dots,-f_k)$, where
\begin{equation} \label{eqn:f-form}
f_i=
\begin{cases}
d_{\ell}      &\mbox{if $i=1$} \\
-d_{\ell+i-1} &\mbox{if $2\le i\le k-\ell+1$} \\
d_{i-k+\ell-1} &\mbox{if $k-\ell+2\le i\le k$.} 
\end{cases}
\end{equation}
Let $F_1,\dots,F_{k-1}$ be a partition of the edges $(\mu,\nu)$ with $\mu^{-1}\nu\leq\delta$, defined by
\begin{equation} \label{eqn:F_j}
F_j:=\{(\mu,\nu) : \mu^{-1}\nu=(f_i,f_{i+j})(-f_i,-f_{i+j}), 1\le i\le k-j\}.
\end{equation}
Then for $1\le j\le k-1$ the following properties hold.
\begin{enumerate}
\item $|F_j|=k-j$.
\item Each $F_j$ contains exactly one good edge.
\item For any $(\mu,\nu)\in F_j$, the reflection $\mu^{-1}\nu$ splits $\delta$ into a paired cycle of size $k-j$ and a paired cycle of size $j$.
\end{enumerate}
\end{lem}

\begin{proof}
(i) That $|F_j|=k-j$ is clear. (ii) By Proposition \ref{pro:nonzero-cycle}, the edge $(\mu,\nu)\in F_j$ with
\begin{equation}
\mu^{-1}\nu=\begin{cases}
(f_{k-\ell+1-j},f_{k-\ell+1})(-f_{k-\ell+1-j},-f_{k-\ell+1}) &\mbox{if $1\le j\le k-\ell-1$}; \\
(f_1,f_{j+1})(-f_1,-f_{j+1}) &\mbox{if $k-\ell\le j\le k-1$}
\end{cases}
\end{equation}
is the only good edge in $F_j$ since $(f_{k-\ell+1-j},f_{k-\ell+1})=(-d_{k-j},-d_k)$ for $1\le j\le k-\ell-1$, $(f_1,f_{k-\ell+1})=(d_{\ell},-d_k)$, and $(f_1,f_{j+1})=(d_{\ell},d_{j-k+\ell})$ for $k-\ell+1\le j\le k-1$.

(iii) For any $(\mu,\nu)\in F_j$, say $\mu^{-1}\nu=(f_i,f_{i+j})(-f_i,-f_{i+j})$, we observe that the cycle $\delta$ is split by $\mu^{-1}\nu$ into a paired cycle $(f_1,\dots,f_i,f_{i+j+1},\dots,f_k)(-f_1,\dots,-f_i,-f_{i+j+1},\dots,-f_k)$ and a paired cycle $(f_{i+1},\dots,f_{i+j})(-f_{i+1},\dots,-f_{i+j})$.
\end{proof}

\subsection{The map $\varphi_B$}
In what follows we shall establish a bijection $\varphi_B$ between the parking functions in $\PF^B$ and the maximal chains in $\NC(B_n,c)$.

Following Lemma \ref{lem:nonzero-cycle-edge-classes}, we assign an order for the edges in each set $F_j$ in (\ref{eqn:F_j}) such that the first edge is the good edge. This will be used in the construction of the map $\varphi_B$. 

\medskip
\noindent
{\bf Procedure C.1}

By abuse of notation, for $1\le j\le k-\ell-1$, let $F_j$ denote the sequence of reflections
\begin{equation*} \label{eqn:linearing-ordering-1}
\begin{aligned}
& (-d_{k-j},-d_k)(d_{k-j},d_k), (-d_{k-j+1},d_1)(d_{k-j+1},-d_1), \dots, (d_{\ell-j-1},d_{\ell-1})(-d_{\ell-j-1},-d_{\ell-1}), \\
     &\quad (d_{\ell},-d_{\ell+j})(-d_{\ell},d_{\ell+j}),(-d_{\ell+1},-d_{\ell+j+1})(d_{\ell+1},d_{\ell+j+1}), \dots, (-d_{k-j-1},-d_{k-1})(d_{k-j-1},d_{k-1}).
\end{aligned}
\end{equation*}
Moreover, for $k-\ell\le j\le k-1$, let $F_j$ denote the sequence of reflections
\begin{equation*} \label{eqn:linearing-ordering-2}
(f_1,f_{j+1})(-f_1,-f_{j+1}),(f_2,f_{j+2})(-f_2,-f_{j+2}),\dots,(f_{k-j},f_k)(-f_{k-j},-f_k);
\end{equation*}
namely,
\begin{align*}
F_{k-\ell}:\, & (d_{\ell},-d_k)(-d_{\ell},d_k),(-d_{\ell+1},d_1)(d_{\ell+1},-d_1), (-d_{\ell+2},d_2)(d_{\ell+2},-d_2), \dots; \\
F_{k-\ell+1}:\, & (d_{\ell},d_1)(-d_{\ell},-d_1),(-d_{\ell+1},d_2)(d_{\ell+1},-d_2),\dots; \\
\vdots\, & \\
F_{k-1}:\, & (d_{\ell},d_{\ell-1})(-d_{\ell},-d_{\ell-1}).
\end{align*}

\begin{exa} \label{exa:nonzero-edge-ordering} {\rm
Following Example \ref{exa:nonzero-edges}, the reflections $\mu^{-1}\nu\leq(2, 3, -7, -8, -9)(-2, -3, 7, 8, 9)$ are partitioned into the following four sequences:
\begin{align*}
F_1:\, & (-8,-9)(8,9), (-9,2)(9,-2), (3,-7)(-3,7), (-7,-8)(7,8); \\
F_2:\, & (-7,-9)(7,9), (-8,2)(8,-2), (3,-8)(-3,8); \\
F_3:\, & (3,-9)(-3,9), (-7,2)(7,-2); \\
F_4:\, & (3,2)(-3,-2).
\end{align*}

}
\end{exa}

\medskip
Given a sequence $1=z_1<\cdots<z_m\le n$ of integers, any $\alpha=(a_1,\dots,a_m)\in \{z_1,\dots,z_m\}^m$ can be regarded as a parking function of type $B$, using $m$ parking spaces numbered by $z_1,\dots,z_m$. Let $\pi=(p_1,\dots,p_m)$ be the outcome of $\alpha$, which is a permutation of $\{z_1,\dots,z_m\}$. Next, we describe a decomposition of $(a_1,\dots,a_{m-1})$ into a parking function of type $B$ (possibly empty) and a parking function of type $A$ with respect to $p_m$. To show this, we construct a specific set $Z_0$ from the prefix $(p_1,\dots,p_{m-1})$ of $\pi$ by the following procedure.

\medskip
\noindent
{\bf Procedure C.2}

For $(p_1,\dots,p_m)$, if $p_m=1$ (i.e., $z_1$), let $Z_0$ be empty.  If $p_m\neq 1$, in the $(m-1)$-tuple $(p_1,\dots,p_{m-1})$ find the greatest integer $z_g$ such that all of $z_1,z_2,\dots,z_{g-1}$ occur to the left of $z_g$, and then find the least integer $z_{m-h+1}$ such that all of $z_m,z_{m-1},\dots,z_{m-h+1}$ occur to the left of $z_g$ if any. Let $Z_0$ denote the set $\{z_1,\dots,z_g,z_{m-h+1},\dots,z_m\}$. We call $Z_0$ the \emph{zeta-naught set} of $(p_1,\dots,p_{m-1})$. Let $Z_1=\{p_1,\dots,p_{m-1}\}\setminus Z_0$.

\medskip
By a subword of $\alpha$ \emph{induced} on $Z_0$ ($Z_1$, respectively) we mean the subword consisting of the entries $a_i$ such that $p_i\in Z_0$ ($Z_1$, respectively).
For example, consider a parking function $\alpha=(3,8,8,7,8,5,5,7)$ in $\PF^B_8$ with outcome $\pi=(p_1,\dots,p_8)=(3,8,1,7,2,5,6,4)$. The zeta-naught set of $(p_1,\dots,p_7)$ is $Z_0=\{1,2,7,8\}$, and then $Z_1=\{3,5,6\}$. The subword of $\alpha$ induced on $Z_0$ ($Z_1$, respectively) is $(8,8,7,8)$ ($(3,5,5)$, respectively).
The following result shows that the subword of $\alpha$ induced on $Z_0$ ($Z_1$, respectively) is a parking function of type $B$ ($A$, respectively).

\medskip
\begin{lem} \label{lem:separation-zeta-set}
For $1\le j\le m-1$, we have
\begin{enumerate}
\item  $p_j\in Z_1$ if and only if $a_j\in Z_1$. In this case, either $a_j\le p_j<p_m$ or $p_m<a_j\le p_j$.
\item  $p_j\in Z_0$ if and only if $a_j\in Z_0$.
\end{enumerate}
\end{lem}

\begin{proof} (i) If $p_m=1$ (i.e., $z_1$) then the space $1$ is last occupied. The set $Z_0$ is empty and $Z_1$ is $\{z_2,\dots,z_m\}$. Note that $a_j\neq 1$ for each $j\in [m-1]$. Thus $a_j\in Z_1$, say $a_j=z_{\ell}$. If $p_j<a_j$, the spaces $z_{\ell},z_{\ell+1},\dots,z_m$ have been occupied when the $j$th car arrives. Then the $j$th car will park at space $z_1$, a contradiction. Thus, $p_m<a_j\le p_j$.  Let $p_m\neq 1$. Suppose the $k$th car parks at space $z_g$, i.e., $z_g=p_k$ for some $k<m$.

($\Longrightarrow$) Suppose $p_j=z_{\ell'}\in Z_1$. If $a_j\in Z_0$, consider the following two cases.
\begin{itemize}
\item $j>k$. The parking spaces $z_{g+1},\dots,z_{\ell'-1}$ have been occupied when the $j$th car arrives. It follows that the elements $z_1,z_2,\dots,z_{\ell'-1}$ occur to the left of $z_{\ell'}$ in $\pi$, which is against the choice of $z_g$. 
\item $j<k$. The space $z_g$ is empty when the $j$th car has parked, and hence the $j$th car parks at one of the spaces in $Z_0$, which is against $p_j\in Z_1$.
\end{itemize}
Thus, we have $a_j\in Z_1$.

($\Longleftarrow$) Suppose $a_j=z_{\ell}\in Z_1$. If $p_j\in Z_0$, by Procedure C.2 we have $j\le k$. We observe that the parking spaces $z_{\ell}, z_{\ell+1}\dots, z_{m-h}$ have been occupied when the $j$th car arrives. It follows that the elements $z_{\ell}, z_{\ell+1}\dots, z_m$ occur to the left of $z_g$ in $\pi$, which is against the choice of $z_{m-h+1}$. Thus, we have $p_j\in Z_1$.

In the case of $a_j,p_j\in Z_1$, if $a_j<p_m$ then the $j$th car will park at a space before $p_m$ since the space $p_m$ is not occupied until the last car arrives. Thus, $a_j\le p_j<p_m$. Otherwise, $a_j>p_m$ and the $j$th car will park at one of the spaces $a_j=z_{\ell}, z_{\ell+1},\dots, z_{m-h}$. Thus $p_m<a_j\le p_j$. The assertion (i) follows.

The assertion (ii) is proved by the same argument.
\end{proof}

We have a further observation about Lemma \ref{lem:separation-zeta-set}(i).

\medskip
\begin{lem} \label{lem:zeta-set-fact}
If $p_m\neq 1$ then the entries $p_j\in Z_1$ satisfying $p_j<p_m$ occur to the left of $z_g$ in $(p_1,\dots,p_{m-1})$.
\end{lem}

\begin{proof} By Procedure C.2, we observe that every element in $Z_1$ is greater than $z_g$. 
Suppose $p_j=z_{\ell}\in Z_1$ is the least entry on the right of $z_g$ in $(p_1,\dots,p_{m-1})$. If $p_j<p_m$ then the elements $z_{g+1},\dots,z_{\ell-1}$ occur to the left of $z_g$. It follows that all of $z_1,z_2,\dots,z_{\ell-1}$ occur to the left of $z_{\ell}$, which is against the choice of $z_g$. Hence $p_j>p_m$. The result follows.
\end{proof}

Now, we describe the map $\varphi_B$. Given $\alpha=(a_1,\dots,a_n)\in\PF^B_n$ with outcome $\pi=(p_1,\dots,p_n)$,
let $\varphi_B(\alpha)=\ww$, where $\ww:e=\omega_0<\omega_1<\cdots<\omega_n=c$ is a maximal chain in $\NC(B_n,c)$, and let $t_1t_2\cdots t_n$ be the corresponding factorization of $c$, where $t_i=\omega^{-1}_{i-1}\omega_i$. We shall construct the reflections $t_1,t_2,\dots, t_n$ in reverse order by the following procedure.
When working on $t_i$, the cycles of $\omega_i$ have been ordered $\delta_0\delta_1\cdots \delta_{n-i}$, where $\delta_0$ is the balanced cycle (possibly empty), and each $\delta_j$, $j>0$, is a paired cycle. The cycle $\delta_0$ ($\delta_j$, respectively) is associated with a tuple $\zeta_0$ ($\zeta_j$, respectively) of integers, where $\zeta_0,\zeta_1,\dots,\zeta_{n-i}$ form a partition of $[n]$. 
Initially, we associate $\omega_n=\delta_0$ with the $n$-tuple $\zeta_0=(1,\dots,n)$. In the following, let $\overline{\ell} \pmod m$ denote the integer congruent to $\ell$ mod $m$ with $1\le \overline{\ell}\le m$.

\medskip
\noindent
{\bf Algorithm C}

For $i=n,\dots, 1$, the reflection $t_i$ is determined by $a_i$ and $p_i$ as follows. 
Among the tuples $\zeta_0,\dots,\zeta_{n-i}$, find the tuple that contains the integer $p_i$. There are two cases.

{\bf Case 1.} $p_i\in\zeta_0$. Suppose $\delta_0=(d_1,\dots,d_{2m})$ for some $m$, where $0<d_1<\cdots<d_{m}$ and $d_{m+i}=-d_i$ for all $i\in [m]$. If $m=1$, we set $t_i=(d_1,-d_1)$. 

Let $m>1$. Let $\zeta_0=(z_1,\dots,z_{m})$, where $1=z_1<\cdots<z_{m}$, say $p_i=z_{\ell}$ and $a_i=z_{\ell'}$. Let $j=\overline{\ell'-\ell} \pmod {m}$. 

If $p_i=1$,  we set $t_i=(d_{m+j},d_{2m+j})$.
Note that $\delta_0$ is split by $t_i$ into a paired cycle $\delta'_1$ of size $m$. We associate $\delta'_1$ with the $m$-tuple $\zeta'_1=(z_2,\dots,z_{m},1^*)$.

Otherwise, $p_i\neq 1$. Let $\pi_0$ be the subword of $\pi$ consisting of the entries $p_r$ that occur in $\zeta_0$ and $r<i$. By Procedure C.2, create the zeta-naught set $Z_0$ of $\pi_0$. Suppose $|Z_0|=m-k$ for some $k\in [m-1]$, we set 
\begin{equation} \label{eqn:t_i}
 t_i=(d_{m-k+j},d_{m+j})(d_{2m-k+j},d_{2m+j}). 
\end{equation}
Note that $\delta_0$ is split by $t_i$ into a balanced cycle $\delta'_0$ of size $2(m-k)$ and a paired cycle $\delta'_1$ of size $k$, written in the order $\delta'_0\delta'_1$. 

Let $\{z_1,\dots,z_{m}\}\setminus Z_0=\{z_{h+1},\dots,z_{h+k}\}$ for some integer $h$. 
Since $p_i\not\in Z_0$, $p_i=z_{h+b}$ for some $b$. 
We associate $\delta'_0$ with an $(m-k)$-tuple $\zeta'_0$ and associate $\delta'_1$ with a $k$-tuple $\zeta'_1$ by setting
\begin{align}  \label{eqn:zeta-sequences-case-0}
\zeta'_0 &=(z_1,\dots,z_h,z_{h+k+1},\dots,z_{m}); \\  \label{eqn:zeta-sequences-case-1}
\zeta'_1 &=(z_{h+b+1},\dots,z_{h+k},z_{h+1},\dots,z_{h+b-1},z_{h+b}^*).
\end{align}
(The entry marked with a star, which equals $p_i$, in the tuple $\zeta'_1$ indicates that the corresponding parking space is not occupied until the $i$th car arrives.)

\smallskip
{\bf Case 2.} $p_i\in\zeta_r$ for some $r>0$. Suppose
\begin{equation*}
\delta_r=(d_1,\dots,d_{\ell},-d_{\ell+1},\dots,-d_{k})(-d_1,\dots,-d_{\ell},d_{\ell+1},\dots,d_{k})
\end{equation*}
for some $k$, where $0<d_1<\cdots<d_{k}$ and $1\le \ell\le k$. Let $\zeta_r=(y_1,\dots,y_{k-1},y^*_{k})$, say $p_i=y_j$ and $a_i=y_{j'}$, for some $j'\le j\le k-1$. (Currently, the parking space $y_k$ is empty.)

Write $\delta_r=(f_1,\dots,f_k)(-f_1,\dots,-f_k)$ as defined in (\ref{eqn:f-form}). Using Procedure C.1, create a family $F_1,\dots,F_{k-1}$ of sequences of reflections from $\delta_r$, and then set $t_i$ to be the $(j-j'+1)$th reflection in the sequence $F_{k-j}$, say $t_i=(f_{b},f_{b+k-j})(-f_{b},-f_{b+k-j})$ for some $b$.
Note that $\delta_r$ is split by $t_i$ into a paired cycle $\delta_{r,1}$ of size $j$ and a paired cycle $\delta_{r,2}$ of size $k-j$, written in the order $\delta_{r,1}\delta_{r,2}$, where
\begin{equation} \label{eqn:paired-cycle-split}
\begin{aligned}  
\delta_{r,1} &=(f_1,\dots,f_b,f_{b+k-j+1},\dots,f_k)(-f_1,\dots,-f_b,-f_{b+k-j+1},\dots,-f_k); \\  
\delta_{r,2} &=(f_{b+1},\dots,f_{b+k-j})(-f_{b+1},\dots,-f_{b+k-j}).
\end{aligned}
\end{equation}
Their associated tuples $\zeta_{r,1}, \zeta_{r,2}$ are obtained by cutting $\zeta_r$ after the entry $p_i$, i.e., 
\begin{equation} \label{eqn:zeta-sequences-case-2}
\zeta_{r,1}=(y_1,\dots,y_{j-1},y^*_{j}); \qquad \zeta_{r,2}=(y_{j+1},\dots,y_{k-1},y_{k}).
\end{equation}

\medskip
\begin{exa} \label{exa:PF-to-noncrossing-B} {\rm
Let $\alpha=(3,8,8,7,8,5,5,7)\in\PF^B_8$ with outcome $(3,8,1,7,2,5,6,4)$. Let $\varphi_B(\alpha)=\ww$, where $\ww:\omega_0<\omega_1<\cdots<\omega_8$, and let $t_1t_2\cdots t_8$ be the corresponding factorization of $\omega_8=(1,\dots,8,-1,\dots,-8)$, where $t_i=\omega_{i-1}^{-1}\omega_i$.
The construction of $t_1,\cdots,t_8$ is shown in Table \ref{tab:construct-type-B-noncrossing}. We use the abbreviation $[d_1,\dots,d_m]$ for a balanced cycle $(d_1,\dots,d_m,-d_1,\dots,-d_m)$, and the abbreviation $(d_1,\dots,d_{\ell},-d_{\ell+1},\dots,-d_{k})$ for a paired cycle in (\ref{eqn:nonzero}). Moreover, the tuple $\zeta_0$ associated to a balanced cycle is underlined, and the good and bad edges are indicated by $\circ$ and $\times$, respectively. Some initial steps are described below.

(1) $a_8=7$ and $p_8=4$. Initially, $p_8$ occurs in $\zeta_0=(1,\dots,8)$. We have $\overline{a_8-p_8}\equiv 3\pmod 8$. Creating the zeta-naught set of the sequence $(3,8,1,7,2,5,6)$, we have $Z_0=\{1,2,7,8\}$ and hence $|Z_0|=4$. By (\ref{eqn:t_i}), we set $t_8=(7,-3)(-7,3)$. Then $\omega_8$ is split by $t_8$ into $\delta'_0, \delta'_1$, where $\delta'_0=(4,5,6,7,-4,-5,-6,-7)$ and $\delta'_1=(1,2,3,-8)(-1,-2,-3,8)$. For their associated tuples, we set $\zeta'_0=(1,2,7,8)$ and $\zeta'_1=(5,6,3,4^*)$.

(2) $a_7=5$ and $p_7=6$. Note that $p_7$ occurs in the tuple $\zeta_1=(y_1,y_2,y_3,y_4)=(5,6,3,4^*)$, which is associated to a paired cycle $\delta_1=(1,2,3,-8)(-1,-2,-3,8)$. By Procedure C.1, create a family $F_1,F_2,F_3$ of sequences of reflections from $\delta_1$:
\begin{align*}
F_1:\, & (3,-8)(-3,8), (-8,1)(8,-1), (1,2)(-1,-2);\\
F_2:\, & (3,1)(-3,-1), (-8,2)(8,-2); \\
F_3:\, & (3,2)(-3,-2).
\end{align*}
Since $a_7=y_1$ and $p_7=y_2$, we set $t_7=(-8,2)(8,-2)$, i.e., the second reflection in $F_2$. Then $\delta_1$ is split by $t_7$ into $\delta_{1,1}, \delta_{1,2}$, where $\delta_{1,1}=(3,-8)(-3,8)$ and $\delta_{1,2}=(1,2)(-1,-2)$. For their associated tuples, we set $\zeta_{1,1}=(5,6^*)$ and $\zeta_{1,2}=(3,4)$.
}
\end{exa}

\begin{table}[ht]
\caption{The construction of the maximal chain $\varphi_B(\alpha)=t_1t_2\cdots t_8$ in Example \ref{exa:PF-to-noncrossing-B}.}
\centering
{\small
\begin{tabular}{c|cccccll|c}
$i$ &  &  $a_i$ & $p_i$  &  $|Z_0|$ & $t_i$  & \multicolumn{1}{c}{$\omega_i$}    & \multicolumn{1}{c|}{$\zeta_0\zeta_1\cdots\zeta_{8-i}$}  & $\circ\times$  \\[2pt]
\hline
  8 &  &  7     &  4     &   4     &    &  $[1,2,3,4,5,6,7,8]$          
     &  $(\underline{1,2,3,4,5,6,7,8})$   &    \\[-6pt]
    &  &        &        &         &  $(7,-3)(-7,3)$  &            
     &     &   $\times$ \\ 
  7 &  &  5     &  6     &        &     &  $[4,5,6,7](1,2,3,-8)$    & $(\underline{1,2,7,8})(5,6,3,4)$ &    \\[-6pt]
    &  &        &        &        &  $(-8,2)(8,-2)$   &       &   &   $\times$  \\ 
  6 &  &  5     &  5     &        &     &  $[4,5,6,7](3,-8)(1,2)$         & $(\underline{1,2,7,8})(5,6)(3,4)$  &     \\[-6pt]
    &  &        &        &        &  $(3,-8)(-3,8)$   &            &    &   $\circ$ \\ 
  5 &  &  8     &  2     &   2    &     &  $[4,5,6,7](3)(8)(1,2)$       & $(\underline{1,2,7,8})(5)(6)(3,4)$    &    \\[-6pt]
    &  &        &        &        &  $(7,-5)(-7,5)$   &          &      &   $\times$ \\ 
  4 &  &  7     &  7     &        &       &  $[6,7](4,5)(3)(8)(1,2)$        &  $(\underline{1,8})(7,2)(5)(6)(3,4)$    &   \\[-6pt]
    &  &        &        &        &  $(4,5)(-4,-5)$     &           &      &   $\circ$ \\ 
  3 &  &  8     &  1     &  0      &   &  $[6,7](4)(5)(3)(8)(1,2)$  &  $(\underline{1,8})(7)(2)(5)(6)(3,4)$   &    \\[-6pt]
    &  &        &        &         &  $(6,-6)$   &     &    &   $\times$ \\ 
  2 &  &  8     &  8    &         &      &  $(6,-7)(4)(5)(3)(8)(1,2)$  &  $(8,1)(7)(2)(5)(6)(3,4)$      &    \\[-6pt]
    &  &        &       &         &  $(6,-7)(-6,7)$   &     &          &   $\circ$ \\   
  1 &  &  3     & 3     &         &      &  $(6)(7)(4)(5)(3)(8)(1,2)$     & $(8)(1)(7)(2)(5)(6)(3,4)$     &    \\[-6pt]
    &  &        &       &         &  $(1,2)(-1,-2)$   &        &       &   $\circ$ \\
  0 &  &        &       &         &      &  $(6)(7)(4)(5)(3)(8)(1)(2)$   & $(8)(1)(7)(2)(5)(6)(3)(4)$     &   
\end{tabular}
}
\label{tab:construct-type-B-noncrossing}
\end{table}

\medskip
\begin{pro} \label{pro:preserving-B} We have
\begin{enumerate}
\item The maximal chain $\varphi_B(\alpha)$ is well-defined.
\item For $1\le i\le n$, $(\omega_{i-1},\omega_i)$ is a good edge in $\varphi_B(\alpha)$ if and only if the $i$th car of $\alpha$ is lucky.
\end{enumerate}
\end{pro}

\begin{proof} (i) By Algorithm C, it suffices to prove that $a_i$ occurs in the same tuple of integers as $p_i$ for each $i$. 

In the initial step, $a_n$ and $p_n$ are in the $n$-tuple $(1,\dots,n)$, $Z_0$ is the zeta-naught set of $(p_1,\dots,p_{n-1})$, and $Z_1=\{p_1,\dots,p_{n-1}\}\setminus Z_0$.
By Lemma \ref{lem:separation-zeta-set}, the subword of $\alpha$ induced on $Z_0$ ($Z_1$, respectively) is a parking function of type $B$ ($A$, respectively).

For $i\le n-1$, suppose $p_i$ occurs in the tuple $\zeta_0=(z_1,\dots,z_m)$ associated to a balanced cycle, the subword $\alpha'$ of $\alpha$ induced on $\{z_1,\dots,z_m\}$ is a parking function of type $B$. Hence $a_i$ occurs in $\zeta_0$. Moreover, in Case 1 of the algorithm, we observe that $\alpha'$ is decomposed into a parking function of type $B$ ($A$, respectively) with parking spaces $\zeta'_0$ in (\ref{eqn:zeta-sequences-case-0}) ($\zeta'_1$ in (\ref{eqn:zeta-sequences-case-1}), respectively).

On the other hand, suppose $p_i$ occur in the tuple $\zeta_r=(y_1,\dots,y_{k-1},y^*_k)$ associated to a paired cycle, the subword $\alpha''$ of $\alpha$ induced on $\{y_1,\dots,y_{k-1}\}$ is a parking function of type $A$.  Hence $a_i\in\{y_1,\dots,y_{k-1}\}$. Moreover, in Case 2 of the algorithm, $\alpha''$ is decomposed into two parking functions of type $A$ with parking spaces $\zeta_{r,1}$ and $\zeta_{r,2}$ in (\ref{eqn:zeta-sequences-case-2}), respectively.

By induction, $a_{i-1}$ occurs in the same tuple as $p_{i-1}$. Hence $\varphi_B(\alpha)$ is well-defined.

(ii) In Case 1,  suppose $p_i$ occurs in the tuple $\zeta_0=(z_1,\dots,z_m)$, say $p_i=z_{\ell}$. Then $a_i=z_{\ell'}$ for some $\ell'\in [m]$. Let $j=\overline{\ell'-\ell}\pmod m$. We set
\begin{equation*}
t_i=\begin{cases}
(d_{m+j},d_{2m+j}) &\mbox{if $p_i=1$}; \\
(d_{m-k+j},d_{m+j})(d_{2m-k+j},_{2m+j}) &\mbox{if $p_i\neq 1$}.
\end{cases}
\end{equation*}
By Proposition \ref{pro:zero-cycle}, $(\omega_{i-1},\omega_i)$ is a good edge if and only if $j=m$ (i.e., $a_i=p_i$), in which case the $i$th car of $\alpha$ is lucky.

In Case 2, suppose $p_i$ occurs in a tuple $\zeta_r=(y_1,\dots,y_{k-1},y_k^*)$, say $p_i=y_j$, we have $a_i=y_{j'}$ for some $1\le j'\le j\le k-1$. Note that there are $j$ reflections in the sequence $F_{k-j}$, and that we set $t_i$ to be the $(j-j'+1)$th reflection in $F_{k-j}$. By Lemma \ref{lem:nonzero-cycle-edge-classes} and Procedure C.1, $(\omega_{i-1},\omega_i)$ is a good edge if and only if $j'=j$, in which case the $i$th car of $\alpha$ is lucky. The result follows.
\end{proof}

The injectivity of the map $\varphi_B$ from the parking functions in $\PF^B_n$ to the maximal chains in $\NC(B_n,c)$ is proved. By (\ref{eqn:B-polynomial}) and (\ref{eqn:PF^B-generating-poly}), these two statistics $\reluc$ and $\bad$ share the same distribution. Thus, Theorem \ref{thm:type-B-PF-NC} follows.
In what follows we do establish the inverse map $\varphi_B^{-1}$ to prove the surjectivity.

\subsection{The map $\varphi_B^{-1}$} 
Given a maximal chain $\ww: e=\omega_0<\omega_1<\cdots<\omega_{n}=c$ in $\NC(B_n,c)$, let $t_1t_2\cdots t_{n}$ be the corresponding factorization of $c$, where $\omega^{-1}_{i-1}\omega_i=t_i$. Let $\varphi_B^{-1}(\ww)=\alpha=(a_1,\dots,a_{n})\in\PF^B_n$ with outcome $\pi=(p_1,\dots,p_n)$. 
The construction of $\alpha$ is described in two procedures. The first procedure is for constructing the subword of $\alpha$ corresponding to a paired cycle, which is the reverse operation in Case 2 of Algorithm C.

\medskip
\noindent
{\bf Procedure D.1}

Given a subword $\delta=t_{i_1}t_{i_2}\cdots t_{i_{k-1}}$ of $\ww$ and a $k$-tuple $\zeta=(z_1,\dots,z_{k-1},z^*_k)$ associated to $\delta$, where $\delta$ is a paired cycle of size $k$ and $1\le i_1<\dots <i_{k-1}<n$, we construct  $\alpha'$ and $\pi'$ such that $\alpha'=(a_{i_1},\dots,a_{i_{k-1}})\in\{z_1,\dots,z_{k-1}\}^{k-1}$ and $\pi'=(p_{i_1},\dots,p_{i_{k-1}})$ is a permutation of $\{z_1,\dots,z_{k-1}\}$ by the following procedure. Suppose
\begin{equation} \label{eqn:delta-nonzero}
\delta=(d_1,\dots,d_{\ell},-d_{\ell+1},\dots,-d_{k})(-d_1,\dots,-d_{\ell},d_{\ell+1},\dots,d_{k}),
\end{equation}
for some integers $0<d_1<\cdots< d_k$ and $1\le\ell\le k$. 

(i) Write $\delta=(f_1,\dots,f_k)(-f_1,\dots,-f_k)$ as defined in (\ref{eqn:f-form}). Using Procedure C.1, create a family $F_1,\dots,F_{k-1}$ of sequences of reflections from $\delta$.
Suppose $t_{i_{k-1}}$ is the $i$th reflection in $F_j$, we set $a_{i_{k-1}}=z_{k-j+1-i}$ and $p_{i_{k-1}}=z_{k-j}$. Note that $\delta$ is split by $t_{i_{k-1}}$ into a paired cycle $\delta_1$ ($\delta_2$, respectively) of size $k-j$ ($j$, respectively), written in the order $\delta_1\delta_2$. We associate $\delta_1$ and $\delta_2$ with a $(k-j)$-tuple $\zeta_1$ and a $j$-tuple $\zeta_2$, respectively, by setting
\begin{equation*}
\zeta_1 =(z_1,\dots,z_{k-j-1},z^*_{k-j}); \qquad
\zeta_2 =(z_{k-j+1},\dots,z_{k-1},z^*_k).
\end{equation*}

(ii) The rest of entries of $\alpha'$ and $\pi'$ can be obtained from $\delta_1$ and $\delta_2$ by iteration of (i) inductively. 

\begin{exa} \label{exa:A-procedure} {\rm
Suppose a paired cycle $\delta=(1,2,3,-8,-9,)(-1,-2,-3,8,9)$ is factorized as $\delta=t_{i_1}t_{i_2}t_{i_3}t_{i_4}$, where $t_{i_1}=(2,-9)(-2,9)$, $t_{i_2}=(3,-8)(-3,8)$, $t_{i_3}=(1,-9)(-1,9)$, and $t_{i_4}=(2,-8)(-2,8)$. Let $\zeta=(z_1,\dots,z^*_5)=(6,7,3,4,5^*)$  be the tuple associated to $\delta$. Using Procedure C.1, create a family $F_1,\dots,F_4$ of sequences of reflections from $\delta$:
\begin{align*}
F_1:\, & (-8,-9)(8,9), (-9,1)(9,-1), (1,2)(-1,-2), (3,-8)(-3,8);\\
F_2:\, & (3,-9)(-3,9), (-8,1)(8,-1), (-9,2)(9,-2); \\
F_3:\, & (3,1)(-3,-1), (-8,2)(8,-2); \\
F_4:\, & (3,2)(-3,-2).
\end{align*}
Note that $t_{i_4}=(-8,2)(8,-2)$ is the second reflection in $F_3$, we set $a_{i_4}=z_1=6$ and $p_{i_4}=z_2=7$. Then $\delta$ is split into $\delta_1, \delta_2$, where $\delta_1=(3,-8)(-3,8)$ and $\delta_2=(1,2,-9)(-1,-2,9)$ associated with the tuples $\zeta_1=(6,7^*)$ and $\zeta_2=(3,4,5^*)$, respectively. 

The construction of the corresponding words $\alpha'=(3,6,3,6)$ and $\pi'=(3,6,4,7)$ is shown in Table \ref{tab:construct-A-procedure}, where the notation for paired cycles is abbreviated and the working $\delta$ is underlined.

}
\end{exa}

\begin{table}[ht]
\caption{The construction of $\alpha'$ and $\pi'$ in Example \ref{exa:A-procedure}.}
\centering
\begin{tabular}{c|ccllccc|c}
$j$ &  &  $t_{i_j}$      & \multicolumn{1}{c}{$\delta$} & \multicolumn{1}{c}{$\zeta$}  &  $a_{i_j}$  &  $p_{i_j}$ & & $\circ\times$  \\[2pt]
\hline
  4 &  &  $(-8,2)(8,-2)$ &  $(\underline{1,2,3,-8,-9})$   &  $(6,7,3,4,5^*)$     &  6 &  7     &       &   $\times$ \\[2pt]
  3 &  &  $(-9,1)(9,-1)$ &  $(3,-8)(\underline{1,2,-9})$  & $(6,7^*)(3,4,5^*)$   &  3 &  4     &       &   $\times$ \\[2pt]
  2 &  &  $(3,-8)(-3,8)$ &  $(\underline{3,-8})(2,-9)(1)$  &  $(6,7^*)(3,4^*)(5^*)$   &  6 &  6     &       &   $\circ$ \\[2pt]  
  1 &  &  $(2,-9)(-2,9)$ &  $(3)(8)(\underline{2,-9})(1)$  &  $(6^*)(7^*)(3,4^*)(5)$  &  3 &  3     &       &   $\circ$ \\[2pt]
    &  &                 &  $(3)(8)(2)(9)(1)$  &  $(6^*)(7^*)(3^*)(4^*)(5^*)$   &    &        &       &  
\end{tabular}
\label{tab:construct-A-procedure}
\end{table}

Now, we present the map $\varphi_B^{-1}$, which is essentially constructed by the reverse operation in Case 1 of Algorithm C. Given a maximal chain $\ww: e=\omega_0<\omega_1<\cdots<\omega_{n}=c$, find the (unique) chain $e=\omega_{j_0}<\omega_{j_1}<\cdots<\omega_{j_{b}}=\omega_n$ satisfying the following conditions. 
\begin{enumerate}
\item[(D1)] For $1\le i\le b$, each $\omega_{j_i}$ contains a balanced cycle.
\item[(D2)] The integer $j_1$ is the least index such that $\omega_{j_1}$ contains a balanced cycle.
\item[(D3)] For $2\le i\le b$, the integer $j_i$ is the least index such that the balanced cycle of $\omega_{j_{i-1}}$ is properly contained in the balanced cycle of $\omega_{j_i}$.
\end{enumerate}
The map $\varphi_B^{-1}$ is described by an inductive procedure for constructing the subword of $\alpha$ corresponding to the balanced cycle of $\omega_{j_i}$ for $i=1,2,\dots,b$.

\medskip
\noindent
{\bf Algorithm D}

(i) For the base step, suppose $\delta=t_{i_1}t_{i_2}\dots t_{i_h}$ is the balanced cycle of $\omega_{j_1}$ for some integers $1\le i_1<i_2<\cdots<i_h=j_1$, we associate $\delta$ with the $h$-tuple $\zeta=(1,n-h+2,n-h+3,\dots,n)$. The corresponding subword $(a_{i_1},a_{i_2},\dots,a_{i_h})$ of $\alpha$ and the subword $(p_{i_1},p_{i_2},\dots,p_{i_h})$ of $\pi$ are constructed as follows. 

Suppose $\delta=(d_1,\dots,d_h,-d_1,\dots,-d_h)$ for some integers $0<d_1<\cdots<d_h$.
If $h=1$, set $a_{i_h}=p_{i_h}=1$ and we are done. Let $h>1$. Since $\delta$ is split by $t_{i_h}$ into a paired cycle $\delta_1$ of size $h$, note that $t_{i_h}=(d_j,-d_j)$ for some $j\in [h]$. 
We set $p_{i_h}=1$ and set 
\begin{equation} \label{eqn:a_ih}
a_{i_h}=\begin{cases}
1 &\mbox{if $t_{i_h}=(d_h,-d_h)$;} \\
n+1-h+j & \mbox{if $t_{i_h}=(d_j,-d_j)$ and $1\le j\le h-1$.}
\end{cases}
\end{equation}
We associate $\delta_1=t_{i_1}t_{i_2}\cdots t_{i_{h-1}}$ with the $h$-tuple $\zeta_1=(n-h+2,n-h+3,\dots,n,1^*)$.
Using $\delta_1$, $\zeta_1$ and Procedure D.1, we obtain the remaining entries $a_{i_1},a_{i_2},\dots,a_{i_{h-1}}$ and $p_{i_1},p_{i_2},\dots,p_{i_{h-1}}$.

(ii) For the inductive step, suppose $\delta=t_{u_1}t_{u_2}\cdots t_{u_{m}}$ is the balanced cycle of $\omega_{j_i}$ ($i\ge 2$) for some integers $1\le u_1<u_2<\cdots<u_m=j_i$, we shall construct the corresponding subword $(a_{u_1},a_{u_2},\dots,a_{u_m})$ of $\alpha$ and the subword $(p_{u_1},p_{u_2},\dots,p_{u_h})$ of $\pi$.

Let $\delta_0=t_{v_1}t_{v_2}\cdots t_{v_{m'}}$ be the balanced cycle of $\omega_{j_{i-1}}$ for some integers $1\le v_1<v_2<\cdots<v_{m'}=j_{i-1}$. By (D3), $\{v_1,\dots,v_{m'}\}\subset\{u_1,\dots,u_m\}$. Assume that the words $(a_{v_1},\dots,a_{v_{m'}})$ and $(p_{v_1},\dots,p_{v_{m'}})$ corresponding to $\delta_0$ have been determined, where $p_{v_{m'}}=\ell$ for some $\ell\ge 1$. Then the $m'$-tuple $\zeta_0$ associated to $\delta_0$ is
\begin{equation} \label{eqn:m-tuple-zeta-0}
\zeta_0=(1,2,\dots,\ell,n-m'+\ell+1,n-m'+\ell+2,\dots,n).
\end{equation}
The remaining entries of $(a_{u_1},a_{u_2},\dots,a_{u_m})$ and $(p_{u_1},p_{u_2},\dots,p_{u_m})$ are constructed as follows.

Suppose $\delta=(d_1,\dots,d_m,-d_1,\dots,-d_m)$, where $0<d_1<\cdots<d_m$.
Note that $\delta$ is split by $t_{u_m}$ into the balanced cycle $\delta_0$ of size $m'$ and a paired cycle $\delta_1$ of size $m-m'$. We determine the integer $j\in [m]$ such that $t_{u_m}$ is written as
\begin{equation} \label{eqn:t_u_m}
t_{u_m}=(d_{m'+j},d_{m+j})(d_{m+m'+j},d_{2m+j}).
\end{equation}
Moreover, $\delta_1$ consists of the reflections $t_r$ such that $r\in\{u_1,\dots,u_{m-1}\}\setminus\{v_1,\dots,v_{m'}\}$, say $\delta_1=t_{u_{i_1}}t_{u_{i_2}}\cdots t_{u_{i_{m-m'-1}}}$, where $1\le i_1<i_2<\cdots<i_{m-m'-1}\le m-1$. 
As an intermediate stage, we associate $\delta_1$ with the $(m-m')$-tuple $\zeta'_1=(1,\dots,m-m'-1,(m-m')^*)$, and construct a pair $(\alpha',\pi')$ from $\delta_1$ and $\zeta'_1$ by using Procedure D.1, where $\alpha'=\big(a_{u_{i_1}},a_{u_{i_2}}\dots, a_{u_{i_{m-m'-1}}}\big)\in [m-m'-1]^{m-m'-1}$ and $\pi'=\big(p_{u_{i_1}},p_{u_{i_2}},\dots, p_{u_{i_{m-m'-1}}}\big)$ is a permutation of $[m-m'-1]$. 

Among the entries of $\pi'$, find the least integer $m-m'-b$ such that all of $m-m'-1, m-m'-2,\dots,m-m'-b$ occur to the left of $p_{v_{m'}}$, i.e., $p_{u_{i_j}}\in\{m-m'-1, m-m'-2, \dots, m-m'-b\}$ implies $u_{i_j}<v_{m'}$ (possibly all of the entries of $\pi'$ occur to the right of $p_{v_{m'}}$, in which case $b=0$). 
With this integer $b$, we obtain from (\ref{eqn:m-tuple-zeta-0}) the $m$-tuple $\zeta$ associated to $\delta$ by setting 
\begin{equation} \label{eqn:m-tuple-zeta-ii}
\begin{aligned}
\zeta  &=(z_1,\dots,z_m)\\
       &=(1,2,\dots,\ell+b+1,n-m+\ell+b+2,n-m+\ell+b+3,\dots,n).
\end{aligned}
\end{equation}
Then we set $p_{u_m}=\ell+b+1$, say $p_{u_m}=z_r$. Using the integer $j$ from (\ref{eqn:t_u_m}), we set $a_{u_m}=z_{r'}$, where $r'=\overline{r+j}\pmod m$.  From (\ref{eqn:m-tuple-zeta-0}) and (\ref{eqn:m-tuple-zeta-ii}), we replace $\zeta'_1$ by the $(m-m')$-tuple
\begin{equation} \label{eqn:m-tuple-zeta-iii}
\zeta_1=(n-m+\ell+b+2,n-m+\ell+b+3,\dots,n-m'+\ell,\ell+1,\dots,\ell+b,(\ell+b+1)^*).
\end{equation}
Using $\delta_1$, $\zeta_1$ and Procedure D.1, we obtain the requested entries $a_{u_{i_1}},a_{u_{i_2}}\dots, a_{u_{i_{m-m'-1}}}$ and $p_{u_{i_1}},p_{u_{i_2}},\dots, p_{u_{i_{m-m'-1}}}$.

\medskip
\begin{exa} \label{exa:B-procedure} {\rm
Let $\ww$ be a maximal chain in $\NC(B_9,c)$ with the corresponding factorization $t_1t_2\cdots t_9$ of $c$ given in Table \ref{tab:construct-B-procedure}, where we find the chain $e=\omega_0<\omega_4<\omega_7<\omega_9$ satisfying the conditions (D1)-(D3). The construction of $\varphi_B^{-1}(\ww)=\alpha=(a_1,\dots,a_9)$ with outcome $\pi=(p_1,\dots,p_9)$ is described below. 

(1) For the initial step, we observe that the balanced cycle of $\omega_4$ is $\delta=t_2t_4=(4,7,-4,-7)$, and hence we associate $\delta$ with $\zeta=(1,9)$. Since $t_4=(4,-4)$, by (\ref{eqn:a_ih}) we set $p_4=1$ and $a_4=9$. Using $\delta_1=t_2$ and $\zeta_1=(9,1^*)$, we obtain $a_2=9$ and $p_2=9$ by Procedure D.1.

(2) The balanced cycle of $\omega_7$ is $\delta=t_2t_4t_5t_7=(4,5,6,7,-4,-5,-6,-7)$, which is split by $t_7$ into $\delta_0=t_2t_4$ and $\delta_1=t_5$. Since $t_7=(4,6)(-4,-6)$, by (\ref{eqn:t_u_m}) we determine the integer $j=3$ (in this case, $m'=2$ and $m=4$).
Note that $\pi'=(p_5)$ is on the right of $p_4$. 
By (\ref{eqn:m-tuple-zeta-ii}), we determine the tuple $\zeta=(z_1,z_2,z_3,z_4)=(1,2,8,9)$ (in this case, $b=0$). We set $p_7=2=z_2$ and set $a_7=z_1=1$ (in this case, $r=2$ and $j=3$). By (\ref{eqn:m-tuple-zeta-iii}), we have $\zeta_1=(8,2^*)$. Using $\delta_1=t_5=(5,6)(-5,-6)$ and $\zeta_1$, we obtain the entries $a_5=8$ and $p_5=8$ by Procedure D.1.

(3) The balanced cycle of $\omega_9$ is $\delta=t_1t_2\cdots t_9=(1,\dots,9,-1,\dots,-9)$, which is split by $t_9$ splits into $\delta_0=t_2t_4t_5t_7$ and $\delta_1=t_1t_3t_6t_8t_9=(1,2,3,-8,-9)(-1,-2,-3,8,9)$. By $t_9=(3,-7)(-3,7)$, we have $j=3$ (in this case, $m'=4$ and $m=9$).
As an intermediate stage, we associate $\delta_1$ with $\zeta'_1=(1,2,3,4,5^*)$.  Using $\delta_1$ and $\zeta'_1$, we construct a pair $(\alpha',\pi')$ by Procedure D.1, where $\alpha'=(a_1,a_3,a_6,a_8)=(3,1,3,1)$ and $\pi'=(p_1,p_3,p_6,p_8)=(3,1,4,2)$. Note that $\{p_1,p_6\}=\{3,4\}$ is the requested maximal subset of $\pi'$ that occurs to the left of $p_7$. By (\ref{eqn:m-tuple-zeta-ii}), we determine the tuple $\zeta=(z_1,z_2,\dots,z_9)=(1,2,\dots,9)$ (in this case $b=2$). We set $p_9=5=z_5$ and set $a_9=z_8=8$ (in this case $r=5$ and $j=3$). By (\ref{eqn:m-tuple-zeta-iii}), replace $\zeta'_1$ by $\zeta_1=(6,7,3,4,5^*)$. Using $\delta_1$ and $\zeta_1$, we construct the entries $(a_1,a_3,a_6,a_8)=(3,6,3,6)$ and $(p_1,p_3,p_6,p_8)=(3,6,4,7)$ by Algorithm D.1, as shown in Example \ref{exa:A-procedure}. Thus, we obtain the corresponding parking function $\alpha=(3,9,6,9,8,3,1,6,8)$ with outcome $\pi=(3,9,6,1,8,4,2,7,5)$.
}
\end{exa}

\begin{table}[ht]
\caption{The construction of $\varphi_B^{-1}(\ww)$ in Example \ref{exa:B-procedure}.}
\centering
\begin{tabular}{c|cllcc|c}
$i$ & $t_i$ &    \multicolumn{1}{c}{$\omega_i$} & \multicolumn{1}{c}{$\zeta$} &    $a_{i_j}$  &  $p_{i_j}$ &  $\circ\times$  \\[2pt]
\hline
  0 &                  &    $(1)(2)(3)(4)(5)(6)(7)(8)(9)$  &                 &      &               &   \\[2pt]  
  1 &  $(2,-9)(-2,9)$  &    $(1)(2,-9)(3)(4)(5)(6)(7)(8)$  &   &   3  &  3    &   $\circ$ \\[2pt] 
  2 &  $(4,-7)(-4,7)$  &   $(1)(2,-9)(3)(4,-7)(5)(6)(8)$  &   &    9  &   9    &   $\circ$ \\[2pt] 
  3 & $(3,-8)(-3,8)$   &    $(1)(2,-9)(3,-8)(4,-7)(5)(6)$  &   &   6  &  6    &   $\circ$ \\[2pt] 
  4 & $(4,-4)$         &    $[4,7](1)(2,-9)(3,-8)(5)(6)$  &  $(1,9)$ &   9   &    1      &          $\times$ \\[2pt] 
  5 & $(5,6)(-5,-6)$   &    $[4,7](1)(2,-9)(3,-8)(5,6)$  &   &     8  &  8    &   $\circ$ \\[2pt]  
  6 & $(1,-9)(-1,9)$   &    $[4,7](1,2,-9)(3,-8)(5,6)$  &   &    3    &  4    &   $\times$ \\[2pt]  
  7 & $(4,6)(-4,-6)$   &     $[4,5,6,7](1,2,-9)(3,-8)$  &  $(1,2,8,9)$    &    1     & 2       &         $\times$ \\[2pt] 
  8 & $(2,-8)(-2,8)$   &    $[4,5,6,7](1,2,3,-8,-9)$  &        &    6     &  7     &  $\times$ \\[2pt]
  9 & $(3,-7)(-3,7)$   &    $[1,2,3,4,5,6,7,8,9]$     &  $(1,2,3,4,5,6,7,8,9)$       &    8  &  5     &        $\times$ 
\end{tabular}
\label{tab:construct-B-procedure}
\end{table}

\medskip
\noindent
{\bf Remarks.} We describe the strategy used in the inductive step of Algorithm D. Note that $\{p_{v_1},\dots,p_{v_{m'}}\}$ is in fact the zeta-naught set $Z_0$ of $(p_{u_1},p_{u_2},\dots, p_{u_{m-1}})$ with $Z_1=\{p_{u_{i_1}},p_{u_{i_2}},\dots$, $p_{u_{i_{m-m'-1}}}\}$. By Procedure C.2, every element in $Z_1$ is greater than $p_{v_{m'}}=\ell$. By Lemma \ref{lem:zeta-set-fact}, the entries $p_j\in Z_1$ satisfying $p_j<p_{u_m}$ occur to the left of $p_{v_{m'}}$. Thus, in order to determine $p_{u_m}$, we need the greatest integer $b$ such that all of $\ell+1,\dots,\ell+b$ in the word $(p_{u_{i_1}},p_{u_{i_2}},\dots, p_{u_{i_{m-m'-1}}})$ occur to the left of $p_{v_{m'}}$ if any. By (\ref{eqn:zeta-sequences-case-1}), these integers are the last $b+1$ entries in the $(m-m')$-tuple $\zeta_1$. Using the temporary $(m-m')$-tuple $\zeta'_1=(1,\dots,m-m'-1,(m-m')^*)$ associated to $\delta_1=t_{u_{i_1}}t_{u_{i_2}}\cdots t_{u_{i_{m-m'-1}}}$, we construct $(\alpha',\pi')$ by Procedure D.1 and find the integer $b$ from the maximal number of entries $m-m'-1, m-m'-2,\dots,m-m'-b$ of $\pi'$ that occur to the left of $p_{v_{m'}}$. Using this integer $b$, we set $p_{u_m}=\ell+b+1$ and determine the  $m$-tuple $\zeta$ in (\ref{eqn:m-tuple-zeta-ii}) and the $(m-m')$-tuple $\zeta_1$ in (\ref{eqn:m-tuple-zeta-iii}).

\medskip
The proof of Theorem \ref{thm:type-B-PF-NC} is completed. The correspondence of the bijection $\varphi_B$ for $n=3$ is listed in Table \ref{tab:map-PF-B3} for reference.

\begin{table}[ht]
\caption{The map $\varphi_B:\alpha\mapsto\ww=t_1t_2t_3$ for all $\alpha\in\PF^B_3$.}
\centering
{\small
\begin{tabular}{cccc|ccc|c}
\hline
 $\alpha$  & & $\pi$  & & & $\ww$ & &  $\circ\times$\\
\hline
  123   & &  123   & & & $[(1,-1)][(1,2)(-1,-2)][(2,3)(-2,-3)]$    & &   \multirow{6}{0.8cm}{$\circ\circ\circ$}   \\[2pt]
  132   & &  132   & & & $[(1,-1)][(2,3)(-2,-3)][(1,3)(-1,-3)]$    & &      \\[2pt]
  213   & &  213   & & & $[(2,3)(-2,-3)][(1,-1)][(1,3)(-1,-3)]$    & &      \\[2pt]
  231   & &  231   & & & $[(2,3)(-2,-3)][(1,3)(-1,-3)][(3,-3)]$    & &      \\[2pt]
  312   & &  312   & & & $[(1,2)(-1,-2)][(2,-2)][(2,3)(-2,-3)]$    & &      \\[2pt]
  321   & &  321   & & & $[(1,2)(-1,-2)][(2,3)(-2,-3)][(3,-3)]$    & &      \\
\hline
  121   & &  123   & & & $[(2,-2)][(2,3)(-2,-3)][(1,-3)(-1,3)]$    & &    \multirow{12}{0.8cm}{$\circ\circ\times$}  \\[2pt]
  122   & &  123   & & & $[(1,-1)][(1,3)(-1,-3)][(1,2)(-1,-2)]$    & &      \\[2pt]
  131   & &  132   & & & $[(3,-3)][(1,2)(-1,-2)][(2,-3)(-2,3)]$    & &      \\[2pt]
  133   & &  132   & & & $[(2,-2)][(1,-3)(-1,3)][(1,-2)(-1,2)]$    & &      \\[2pt]
  211   & &  213   & & & $[(1,-3)(-1,3)][(2,-2)][(1,-2)(-1,2)]$    & &      \\[2pt]
  212   & &  213   & & & $[(1,2)(-1,-2)][(3,-3)][(2,-3)(-2,3)]$    & &      \\[2pt]
  232   & &  231   & & & $[(1,-2)(-1,2)][(2,3)(-2,-3)][(1,-1)]$      & &      \\[2pt]
  233   & &  231   & & & $[(1,2)(-1,-2)][(2,-3)(-2,3)][(2,-2)]$    & &      \\[2pt]
  311   & &  312   & & & $[(1,3)(-1,-3)][(3,-3)][(1,2)(-1,-2)]$    & &      \\[2pt]
  313   & &  312   & & & $[(2,3)(-2,-3)][(3,-3)][(1,-3)(-1,3)]$    & &      \\[2pt]
  322   & &  321   & & & $[(2,3)(-2,-3)][(1,-3)(-1,3)][(1,-1)]$    & &      \\[2pt]
  323   & &  321   & & & $[(1,-3)(-1,3)][(1,2)(-1,-2)][(2,-2)]$    & &      \\
\hline  
  113   & &  123   & & & $[(2,-2)][(1,-2)(-1,2)][(2,3)(-2,-3)]$    & &   \multirow{3}{0.8cm}{$\circ\times\circ$}   \\[2pt]
  221   & &  231   & & & $[(1,3)(-1,-3)][(1,2)(-1,-2)][(3,-3)]$    & &      \\[2pt]
  332   & &  312   & & & $[(1,-2)(-1,2)][(1,-1)][(2,3)(-2,-3)]$    & &      \\
\hline 
  111   & &  123   & & & $[(3,-3)][(2,-3)(-2,3)][(1,-3)(-1,3)]$    & &    \multirow{6}{0.8cm}{$\circ\times\times$}  \\[2pt]
  112   & &  123   & & & $[(3,-3)][(1,-3)(-1,3)][(1,2)(-1,-2)]$    & &      \\[2pt]
  222   & &  231   & & & $[(1,-3)(-1,3)][(1,-2)(-1,2)][(1,-1)]$    & &      \\[2pt]
  223   & &  231   & & & $[(2,-3)(-2,3)][(1,-3)(-1,3)][(2,-2)]$    & &      \\[2pt]
  333   & &  312   & & & $[(2,-3)(-2,3)][(2,-2)][(1,-3)(-1,3)]$    & &      \\[2pt]
  331   & &  312   & & & $[(1,-3)(-1,3)][(1,-1)][(1,2)(-1,-2)]$    & &      \\
\hline    
\end{tabular}
}
\label{tab:map-PF-B3}
\end{table}

Let $B_n(q):=M(B_n,q)$. We observe that the polynomial $B_n(q)$ satisfies the following recurrence relation.

\begin{pro} 
For $n\ge 1$, we have
\begin{equation*}
B_n(q)=\big(1+(n-1)q\big)\sum_{k=1}^{n} \binom{n-1}{n-k} B_{n-k}(q)A_{k-1}(q),
\end{equation*}
with  $B_0(q)=1$.
\end{pro}

\begin{proof} We enumerate the factorizations $t_1t_2\cdots t_n$ of $c$ corresponding to a maximal chain $\ww:e=\omega_0\le\omega_1\le\cdots\le\omega_n=c$, where $t_i=\omega^{-1}_{i-1}\omega_i$. For congruence of integers modulo $2n$, the notation $n+i$ ($2n+i$, respectively) is used to stand for $-i$ ($i$, respectively) for $1\le i\le n$.

Let $T_1,\dots,T_n$ be a family of sets of reflections given by
\begin{equation} \label{eqn:T-reflection}
\begin{aligned}
T_j &:=\{(i, i+j)(n+i,n+i+j) : 1\le i\le n\} \quad\mbox{if $1\le j\le n-1$;} \\
T_n &:=\{(i, i+n) : 1\le i\le n\}.
\end{aligned}
\end{equation}
The reflection $t_n$ is in one of the sets $T_1,\dots,T_n$, say $t_n\in T_k$. By Lemma \ref{lem:zero-cycle-edge-classes}, among the $n$ reflections in $T_k$, there is exactly one possibility for $t_n$ such that $(\omega_{n-1},\omega_n)$ is a good edge. Moreover, the reflection $t_n$ splits $\omega_n$ into a balanced cycle $\delta_0$ of size $2(n-k)$ and a paired cycle $\delta_1$ of size $k$, and among $t_1, t_2, \dots, t_{n-1}$ there are $n-k$ ($k-1$, respectively) reflections comparable with $\delta_0$ ($\delta_1$, respectively). The generating $q$-polynomial for the maximal chains in the $\delta_0$-noncrossing partition lattice ($\delta_1$-noncrossing partition lattice, respectively) is $B_{n-k}(q)$ ($A_{k-1}(q)$, respectively).  The result follows.
\end{proof}

\section{Words of types $A$ and $B$}  In this section we study a simple object in connection with parking functions from the prospective of the polynomials $M(W,q)$. In the case of $W=B_n$, this object will be used in the proof of the $\gamma$-positivity of $M(B_n,q)$ (Theorem \ref{thm:gamma-positive}).

\subsection{Type-$A$ words} Let $\X^A_n$ denote the set of words $\xx=x_1x_2\cdots x_{n-1}$ over the alphabet $\{0,1,\dots,n\}$. Notice that a word $\xx\in\X^A_n$ consists of $n-1$ letters.
A letter $x_i$ is an \emph{excedance} of $\xx$ if $x_i>i$. Let $\exc(\xx)$ denote the number of excedances of $\xx$. By (\ref{eqn:A-polynomial}), it is easy to see that 
\begin{equation}
M(\mathfrak{S}_{n+1},q)=\sum_{\xx\in\X^A_n} q^{\exc(\xx)}.
\end{equation}

\begin{thm} \label{thm:bijection-PF-Word} There is a bijection $\psi_A: \alpha\mapsto \xx=x_1\cdots x_{n-1}$ of $\PF^A_n$ onto $\X^A_n$ such that for each $i\in\{2,\dots,n\}$ the $i$th car of $\alpha$ is reluctant if and only if $x_{n-i+1}$ is an excedance. 
\end{thm}

The map $\psi_A:\PF^A_n\rightarrow\X^A_n$ is described below. We make use of a cycle with vertex set $V=\{0,1,\dots,n\}$ arranged in increasing order (clockwise). 
Given $\alpha=(a_1,\dots,a_n)\in \PF^A_n$ with outcome $(p_1,\dots,p_n)$, we construct the corresponding word $\psi_A(\alpha)=x_1\cdots x_{n-1}$ from right to left by the following procedure. When working on $x_{n-k+1}$, we first create a bijective labeling $f:V\rightarrow\{0,1,\dots,n\}$ such that $(f(p_1),f(p_2),\dots,f(p_{k-1}))=(n,n-1,\dots,n-k+2)$.

\medskip
\noindent
{\bf Algorithm E1}

For $k=2,\dots,n$, the entry $x_{n-k+1}$ is determined as follows.
\begin{enumerate}
\item  Find the vertices $p_1,\dots, p_{k-1}$ in the cycle, and create a vertex labeling by setting $f(p_i)=n-i+1$ for each $i\in [k-1]$. Then assign the labels $0,1,\dots,n-k+1$ clockwise to the remaining vertices, staring from the first unlabeled vertex after $p_1$.
\item Set $x_{n-k+1}$ to be the label that the vertex $a_k$ receives, i.e., $x_{n-k+1}=f(a_k)$.
\end{enumerate}

\medskip
\begin{exa} \label{exa:type-A-PF-to-word} {\rm
Consider the cars $C_1,\dots,C_6$ with preference list $\alpha=(2,4,2,3,5,1)\in\PF^A_6$. The outcome of $\alpha$ is $(2,4,3,5,6,1)$. As shown in Table \ref{tab:construct-type-A-word}, we have $\psi_A(\alpha)= x_1\cdots x_5=1\,3\,4\,6\,1$. Note that the excedances $x_4$, $x_3$ and $x_2$ of $\psi_A(\alpha)$ correspond to the reluctant cars $C_3$, $C_4$ and $C_5$, respectively.  Some initial steps are described below.

(1) $k=2$. We set $f(2)=6$. Then assign the labels $0,1,2,3,4,5$ to the vertices $3,4,5,6,0,1$, accordingly. By $a_2=4$, we set $x_5=f(4)=1$. Note that $C_2$ is lucky and the letter $x_5$ is not an excedance of $\psi_A(\alpha)$.

(2) $k=3$. We set $f(2)=6$ and $f(4)=5$. Then assign the labels $0,1,2,3,4$ to the vertices $3,5,6,0,1$, accordingly. By $a_3=2$, we set $x_4=f(2)=6$. Note that $C_3$ is reluctant and the letter $x_4$ is an excedance of $\psi_A(\alpha)$.
}
\end{exa}

\begin{table}[ht]
\caption{The construction of the word $\psi_A(\alpha)$ in Example \ref{exa:type-A-PF-to-word}.}
\centering
\begin{tabular}{c|cc|ccccccc|cc}
$k$ & $a_k$  & $p_k$  & $f(0)$ & $f(1)$ & $f(2)$ & $f(3)$ & $f(4)$ & $f(5)$ &  $f(6)$        & $x_{7-k}$ \\
\hline
  1  &  2 &  2 &   &   &   &   &   &     &     &      \\
  2  &  4 &  4 & 4 & 5 & 6 & 0 & 1 &  2  &  3  &     1 \\
  3  &  2 &  3 & 3 & 4 & 6 & 0 & 5 &  1  &  2  &     6 \\
  4  &  3 &  5 & 2 & 3 & 6 & 4 & 5 &  0  &  1  &     4 \\
  5  &  5 &  6 & 1 & 2 & 6 & 4 & 5 &  3  &  0  &     3 \\
  6  &  1 &  1 & 0 & 1 & 6 & 4 & 5 &  3  &  2  &     1  
\end{tabular}
\label{tab:construct-type-A-word}
\end{table}

\begin{lem}
For $2\le k\le n$, the letter $x_{n-k+1}$ is an excedance of $\psi_A(\alpha)$ if and only if the $k$th car of $\alpha$ is reluctant.
\end{lem}

\begin{proof}
Note that the $k$th car of $\alpha$ is reluctant if and only if $a_k\in\{p_1,\dots,p_{k-1}\}$. By the vertex labeling in (i) of Algorithm E1, we have $\{f(p_1),\dots,f(p_{k-1})\}=\{n,\dots,n-k+2\}$.  By setting $x_{n-k+1}=f(a_k)$, we have $x_{n-k+1}>n-k+1$ if and only if $a_k\in\{p_1,\dots,p_{k-1}\}$.
\end{proof}

\smallskip
To find $\psi^{-1}_A$, given $\xx=x_1\cdots x_{n-1}\in\X^A_n$, we shall construct $\psi^{-1}_A(\xx)=(a_1,\dots,a_n)$ in two stages.  We also make use of the cycle with vertex set $V=\{0,1,\dots,n\}$.
In the first stage, using $n+1$ parking spaces numbered $0,1,\dots,n$ arranged in a circle, we construct a parking function $\alpha'=(a'_1,\dots,a'_n)\in\{0,1,\dots,n\}^n$ with outcome $\pi'=(p'_1,\dots,p'_n)$ such that $a'_1=p'_1=0$.   
When working on $a'_k$ and $p'_k$ for $k\ge 2$, a bijective labeling $g:V\rightarrow\{0,1,\dots,n\}$ such that $(g(n),g(n-1),\dots,g(n-k+2))=(p'_1,\dots,p'_{k-1})$ has been created. There will be a unique empty space. In the second stage, we make an adjustment so that the parking space 0 is empty.

\medskip
\noindent
{\bf Algorithm E2}

(i)  Let $a'_1=p'_1=0$. The initial vertex labeling is given by $g(i)=i+1\pmod {n+1}$ for all $0\le i\le n$. This gets $g(n)=0=p'_1$. 
\begin{itemize}
\item For $k\ge 2$, set $a'_k=g(x_{n-k+1})$, say $g(x_{n-k+1})=t$. If $x_{n-k+1}\le n-k+1$ then set $p'_k=t$. Otherwise, $x_{n-k+1}> n-k+1$ and set $p'_k$ to be the next available space, i.e., find the least positive integer $s$ such that $\ell=t+s\pmod {n+1}$ and $g^{-1}(\ell)\le n-k+1$, and then set $p'_k=\ell$.
\item Set $g(n-k+1)=p'_k$ and relabel the vertices $0,1,\dots,n-k$ in increasing order by the integers $\{0,1,\dots,n\}\setminus\{g(n-k+1),\dots,g(n)\}$.
\end{itemize}

(ii) Find the unique element in $\{0,1,\dots,n\}\setminus\{p'_1,\dots,p'_n\}$, say $b>0$. Then the corresponding parking function $\psi^{-1}_A(\xx)=(a_1,\dots,a_n)$ is obtained by setting $a_i\equiv a'_i+n+1-b\pmod {n+1}$.

\medskip
\begin{exa} \label{exa:type-A-PF-to-word-inverse} {\rm
Let $\xx= 1\,3\,4\,6\,1\in \X^A_6$. As shown in Table \ref{tab:construct-type-A-word-inverse}, we obtain $\alpha'=(0,2,0,1,3,6)$ with outcome $\pi'=(0,2,1,3,4,6)$, where the empty space is numbered by 5. Hence  the corresponding parking function of $\xx$ is $\psi^{-1}_A(\xx)=(2,4,2,3,5,1)$.
}
\end{exa}

\begin{table}[ht]
\caption{The construction of the first stage $\alpha'$ in Example \ref{exa:type-A-PF-to-word-inverse}.}
\centering
\begin{tabular}{c|c|cccccccc|ccc}
$k$ &  $x_{7-k} $ & $g(0)$ & $g(1)$ & $g(2)$ & $g(3)$ & $g(4)$ & $g(5)$ &  $g(6)$ &  &   & $a'_k$ & $p'_k$    \\
\hline
  1  &   &    &   &   &   &   &   &     &  &  &  0  &  0 \\
  2  & 1 &  1 & 2 & 3 & 4 & 5 & 6 &  0  &  &  &  2  &  2\\
  3  & 6 &  1 & 3 & 4 & 5 & 6 & 2 &  0  &  &  &  0  &  1\\
  4  & 4 &  3 & 4 & 5 & 6 & 1 & 2 &  0  &  &  &  1  &  3\\
  5  & 3 &  4 & 5 & 6 & 3 & 1 & 2 &  0  &  &  &  3  &  4\\
  6  & 1 &  5 & 6 & 4 & 3 & 1 & 2 &  0  &  &  &  6  &  6
\end{tabular}
\label{tab:construct-type-A-word-inverse}
\end{table}

\smallskip
The proof of Theorem \ref{thm:bijection-PF-Word} is completed.

\subsection{Type-$B$ words} Let $\X^B_n$ denote the set of words $\xx=x_1x_2\cdots x_{n}$ over the alphabet $\{1,\dots,n\}$.
Likewise, a letter $x_i$ is an excedance of $\xx$ if $x_i>i$, and the number of excedances of $\xx$ is $\exc(\xx)=\#\{x_i : x_i>i, 1\le i\le n\}$. By (\ref{eqn:B-polynomial}), it is easy to see that 
\begin{equation} \label{eqn:gf-exc-XB}
M(B_n,q)=\sum_{\xx\in\X^B_n} q^{\exc(\xx)}.
\end{equation}

\begin{thm} \label{thm:bijection-type-B-PF-Word}  There is a bijection $\psi_B: \alpha\mapsto \xx=x_1\cdots x_n$ of $\PF^B_n$ onto $\X^B_n$ such that for each $i\in\{1,\dots,n\}$ the $i$th car of $\alpha$ is reluctant if and only if $x_{n-i+1}$ is an excedance. 
\end{thm}

\smallskip
We describe the map $\psi_B:\PF^B_n\rightarrow\X^B_n$. Given $\alpha=(a_1,\dots,a_n)\in \PF^B_n$ with outcome $(p_1,\dots,p_n)$, the corresponding word $\psi_B(\alpha)=x_1x_2\cdots x_n$ is constructed from right to left by the following procedure.

\smallskip
\noindent
{\bf Algorithm F1}

\begin{enumerate}
\item Set $x_n=a_1$. 
\item For $k\ge 2$, if $a_k=p_k$ then set $x_{n-k+1}=j$, where $p_k$ is the $j$th smallest element among $p_k,p_{k+1},\dots,p_n$. Otherwise, $a_k=p_{\ell}$ for some $\ell\in [k-1]$, and set $x_{n-k+1}=n-\ell+1$. 
\end{enumerate}

\smallskip

\begin{exa} \label{exa:type-B-PF-to-word} {\rm
Consider the cars $C_1,\dots,C_6$ with preference list $\alpha=(2,2,5,5,6,6)\in\PF^B_6$. The outcome of $\alpha$ is $(2,3,5,6,1,4)$. As shown in Table \ref{tab:construct-type-B-word}, the corresponding word is $\psi_B(\alpha)=x_1\cdots x_6=3\,3\,4\,3\,6\,2$.  Note that the excedances $x_5,x_3,x_2,x_1$ of $\psi_B(\alpha)$  correspond to the reluctant cars $C_2,C_4,C_5,C_6$, accordingly.
}
\end{exa}

\begin{table}[ht]
\caption{The construction of the word $\psi_B(\alpha)$ in Example \ref{exa:type-B-PF-to-word}.}
\centering
\begin{tabular}{cc|cccccc}
$k$   &     &  1  &  2  & 3  & 4  &  5  &  6   \\
\hline
$a_k$ &     &  2  &  2  & 5  & 5  &  6  &  6   \\
$p_k$ &     &  2  &  3  & 5  & 6  &  1  &  4   \\
$x_{7-k}$ & &  2  &  6  & 3  & 4  &  3  &  3   \\
\end{tabular}
\label{tab:construct-type-B-word}
\end{table}

\begin{lem} 
For $1\le k\le n$, the letter $x_{n-k+1}$ is an excedance of $\psi_B(\alpha)$ if and only if the $k$th car of $\alpha$ is reluctant.
\end{lem}

\begin{proof}
It is obvious that $x_n$ is not an excedance of $\xx$, just as the first car of $\alpha$ is always lucky. For $k\ge 2$, the $k$th car $C$ of $\alpha$ is lucky if $a_k=p_k$. In this case, by (ii) of Algorithm F1, $x_{n-k+1}\le |\{p_k,\dots,p_n\}|= n-k+1$. On the other hand, $C$ is reluctant if $a_k=p_{\ell}$ for some $\ell\in [k-1]$. In this case, $x_{n-k+1}=n-\ell+1>n-k+1$. Thus, the result follows. 
\end{proof}

\smallskip
To find $\psi^{-1}_B$, given a word $\xx= x_1\cdots x_n\in\X^B_n$, we construct $\psi^{-1}_B(\xx)=(a_1,\dots,a_n)$ by the following procedure. 
Recall that $\overline{j} \pmod n$ stands for the integer congruent to $j$ mod $n$ with $1\le \overline{j}\le n$.

\smallskip
\noindent
{\bf Algorithm F2}

\begin{enumerate}
\item Set $a_1=p_1=x_n$. 
\item For $k\ge 2$, if $x_{n-k+1}$ is not an excedance of $\xx$, say $x_{n-k+1}=j\le n-k+1$, we set $a_k$ to be the $j$th smallest element in $\{1,\dots,n\}\setminus \{p_1,\dots,p_{k-1}\}$, and set $p_k=a_k$.
Otherwise, $x_{n-k+1}$ is an excedance, say $x_{n-k+1}=\ell>n-k+1$. Then we set $a_k=p_{n-\ell+1}$, say $p_{n-\ell+1}=t$, and set $p_k$ to be the next available space, i.e., find the least integer $s$ such that $r=\overline{t+s}\pmod n$ and $r\in\{1,\dots,n\}\setminus\{p_1,\dots,p_{k-1}\}$, and then set $p_k=r$.
\end{enumerate}

\smallskip
The proof of Theorem \ref{thm:bijection-type-B-PF-Word} is completed.

\section{Gamma-positivity of $M(B_n,q)$}
In this section we prove the $\gamma$-positivity of the polynomial $M(B_n,q)$ in terms of excedance number of the words in $\X^B_n$. 

\begin{thm} \label{thm:gamma-positive-for-words} For all $n\ge 1$, we have
\begin{equation} \label{eqn:gamma-nj-coefficient}
M(B_n,q)= \sum_{j=0}^{\lfloor (n-1)/2\rfloor} \gamma_{n,j} q^j(1+q)^{n-1-2j},
\end{equation}
where $\gamma_{n,j}$ is the number of words $\xx=x_1\cdots x_n\in\X^B_n$ with $\exc(\xx)=j$ such that $x_1=1$ and $\xx$ contains no two consecutive excedances.
\end{thm}

The generating function for the $\gamma$-coefficients $\gamma_{n,j}$ can be expressed as follows.

\begin{lem} \label{lem:gamma-nj}
We have
\begin{equation} \label{eqn:gamma-nj}
\sum_{j=0}^{\lfloor (n-1)/2\rfloor} \gamma_{n,j} q^j=\begin{cases}
{\displaystyle n\cdot\prod_{k=1}^{\frac{n-1}{2}} \big(k(n-k)+(n-2k)^2q\big)} &\mbox{if $n$ is odd;} \\[3ex]
{\displaystyle \frac{n^2}{2}\cdot\prod_{k=1}^{\frac{n-2}{2}} \big(k(n-k)+(n-2k)^2q\big)} &\mbox{if $n$ is even.} 
\end{cases}
\end{equation}
\end{lem}

\begin{proof}
Using (\ref{eqn:B-polynomial}), express $M(B_n,q)$ as a product of the following factors 
\begin{align*}
\big((n-k)+kq\big)\big(k+(n-k)q\big) 
&= k(n-k)+\big(k^2+(n-k)^2\big)q+k(n-k)q^2 \\
&= k(n-k)(1+q)^2+(n-2k)^2q.
\end{align*}
We have
\begin{align*}
M(B_n,q) &= \prod_{k=0}^{n-1} \big((n-k)+kq \big) \\
&=n\cdot\prod_{k=1}^{n-1} \big((n-k)+kq\big) \\
&=\begin{cases}
{\displaystyle n\cdot\prod_{k=1}^{\frac{n-1}{2}} \big((n-k)+kq\big)\big(k+(n-k)q\big)} &\mbox{if $n$ is odd;} \\[2ex]
{\displaystyle n\left(\frac{n}{2}+\frac{nq}{2}\right)\prod_{k=1}^{\frac{n-2}{2}} \big((n-k)+kq\big)\big(k+(n-k)q\big)} &\mbox{if $n$ is even.} 
\end{cases} \\
&=\begin{cases}
{\displaystyle n\cdot\prod_{k=1}^{\frac{n-1}{2}} \big(k(n-k)(1+q)^2+(n-2k)^2q\big)} &\mbox{if $n$ is odd;} \\[2ex]
{\displaystyle \frac{n^2}{2}(1+q)\prod_{k=1}^{\frac{n-2}{2}} \big(k(n-k)(1+q)^2+(n-2k)^2q\big)} &\mbox{if $n$ is even.} 
\end{cases}
\end{align*}
Thus, the result follows from collecting the terms with $q^j(1+q)^{n-1-2j}$.
\end{proof}

Let $\F_{n}\subset \X^B_n$ be the set of words $\xx=x_1\cdots x_n$ such that $x_1=1$ and $\xx$ contains no two consecutive excedances. We present a determinantal expression of the generating function for the words in $\X^B_n$ ($\F_{n}$, respectively) with respect to excedance numbers, using a method in \cite[Eq.\,(12.4)]{Gessel-Stanley}.

\medskip
\begin{lem} \label{lem:gf-Fnk} For $1\le j\le n-1$, let $m_j=(n-j)q/j$. The following results hold.
\begin{enumerate}
\item We have
\begin{equation} \label{eqn:right-hand-side}
M(B_n,q)= n!\cdot
\det
\begin{bmatrix}
m_1  & m_2  & m_3  &  \cdots & m_{n-1} & 1 \\
-1   & m_2  & m_3  &  \cdots & m_{n-1} & 1 \\
   & -1   & m_3    &  \cdots & m_{n-1} & 1 \\
   &      & \ddots  & \ddots     & \vdots  & \vdots \\
   &      &      &   -1     & m_{n-1}  &  1 \\
   &      &      &          & -1 &  1   
\end{bmatrix}.
\end{equation}
\item We have
\begin{equation} \label{eqn:determinant-form}
\sum_{\xx\in \F_{n}} q^{\exc(\xx)}=
n!\cdot
\det
\begin{bmatrix}
0  & m_2  & m_3  &  \cdots & m_{n-1} & 1 \\
-1 & 0    & m_3  &  \cdots & m_{n-1} & 1 \\
   & -1   & 0    &  \cdots & m_{n-1} & 1 \\
   &      & \ddots  & \ddots     & \vdots  & \vdots \\
   &      &      &   -1     & 0  &  1 \\
   &      &      &          & -1 &  1   
\end{bmatrix}.
\end{equation}
\end{enumerate}
\end{lem}

\begin{proof}
(i) For any $A\subset [n-1]$, let $f(A)$ be the number of words $\xx\in\X^B_n$ with excedance set $A$. If $A=\{b_1< b_2<\cdots< b_\ell\}$, we observe that
\begin{equation} \label{eqn:g(A)}
f(A)=n!\cdot\prod_{j=1}^{\ell} \left(\frac{n-b_j}{b_j}\right).
\end{equation}

Suppose that $M=(m_{ij})$ is an $n\times n$ matrix for which $m_{ij}=0$ if $j<i-1$. By the permutation expansion for determinants, if $\prod_{i=1}^n m_{i,\sigma(i)}\neq 0$ then every cycle of $\sigma$ is of the form $(b, b-1,\dots, a+1, a)$. Moreover, if $m_{i,i-1}=1$ for $2\le i\le n$ then the contribution to $\det M$ from the permutation $\sigma=(b_1,b_1-1,\dots,2,1)(b_2,\dots,b_1+1)\cdots(b_{\ell+1},\dots,b_{\ell}+1)$, where $1\le b_1<b_2<\cdots<b_{\ell+1}=n$, is $(-1)^{n-\ell-1} m_{1,b_1}m_{b_1+1,b_2}\cdots m_{b_{\ell}+1,n}$. Thus, if we set $m_{i,i-1}=-1$ for $2\le i\le n$, $m_{ij}=(n-j)q/j$ for $1\le i\le j\le n-1$ and $m_{i,n}=1$ for $1\le i\le n$, we derive
\begin{equation} \label{eqn:det(M)}
\begin{aligned} 
n!\cdot\det M &=\sum_{A\subset [n-1]} f(A)q^{|A|} \\
       &=\sum_{\xx\in\X^B_n} q^{\exc(\xx)}.
\end{aligned}
\end{equation}
By (\ref{eqn:gf-exc-XB}) and (\ref{eqn:det(M)}), the assertion (i) follows. 

(ii) If we set further $m_{i,i}=0$ for $1\le i\le n-1$ then the excedance sets $A=\{b_1<b_2<\cdots<b_{\ell}\}\subset [n-1]$ that contribute $m_{1,b_1}m_{b_1+1,b_2}\cdots m_{b_{\ell}+1,n}\neq 0$ to $\det M$ satisfy $b_1>1$ and $b_{j+1}-b_j>1$ for each $j\in [\ell-1]$. Thus, the assertion (ii) follows.
\end{proof}

\medskip
To evaluate the determinant in (\ref{eqn:determinant-form}), we make use of Chu's formula for the determinants of a family of tridiagonal matrices \cite{Chu}.

\medskip
\begin{thm}[Chu] \label{thm:Chu} Let
\begin{equation} \label{eqn:Chu-matrix}
f_n(x,y;u,v):=\det
\begin{bmatrix}
x  & u      &  0     &  0  & \cdots  &    0    & 0  & 0\\
nv & x+y    & 2u     &  0  & \cdots  &    0    & 0  & 0\\
0  & (n-1)v & x+2y   &  3u & \cdots  &    0    & 0  & 0\\
\vdots & \vdots & \vdots   & \ddots  &   \ddots &  \ddots & \vdots      & \vdots \\
0  & 0      &    0    &  0  &\cdots  & 2v  &  x+(n-1)y & nu \\
0  & 0      &    0    &  0  &\cdots  &  0  &  v        & x+ny   
\end{bmatrix}.
\end{equation}
We have 
\begin{equation}
f_n(x,y;u,v)=\prod_{k=0}^{n} \left(x+\frac{ny}{2}+\frac{n-2k}{2}\sqrt{y^2+4uv} \right).
\end{equation}
\end{thm}

\medskip
To prove Theorem \ref{thm:gamma-positive-for-words}, it suffices to show that the generating function for $\gamma_{n,j}$ in Lemma \ref{lem:gamma-nj} coincides with the polynomial in Lemma \ref{lem:gf-Fnk}(ii).

\medskip
\noindent
\emph{Proof of Theorem \ref{thm:gamma-positive-for-words}.}  By Lemma \ref{lem:gf-Fnk}(ii), we have
\begin{align}
\sum_{\xx\in \F_{n}} q^{\exc(\xx)}&=n!\cdot
\det
\begin{bmatrix}
1  & m_2  &      &    &         &   \\
-1 & 1    & m_3  &    &         &   \\
   & -1   & 1    & \ddots     &    &   \\
   &      & -1   & \ddots     & m_{n-2}   &   \\
   &      &      & \ddots     & 1  &  m_{n-1} \\
   &      &      &            & -1 &  1   
\end{bmatrix}_{(n-1)\times(n-1)} \label{eqn:D3} \\
&=n\cdot
\det
\begin{bmatrix}
1  & (n-2)q  &           &         &         &\\
-1 & 2       & (n-3)q    &         &         &   \\
   & -2      & 3         & \ddots  &         &   \\
   &         & -3        & \ddots  &   2q    &   \\
   &         &           & \ddots  &  n-2    &  q \\
   &         &           &         & -(n-2)  &  n-1   
\end{bmatrix}. \label{eqn:D4}
\end{align}
The matrix in (\ref{eqn:D3}) is obtained from the matrix in (\ref{eqn:determinant-form}) by subtracting the $(i+1)$th row from the $i$th row for each $i\in [n-1]$ and expanding the determinant along the last column. The matrix in (\ref{eqn:D4}) is obtained from the matrix in (\ref{eqn:D3}) by multiplying the $j$th column by $j$ for each $j\in [n-1]$.
Note that the transpose of the matrix in (\ref{eqn:D4}) coincides with the tridiagonal matrix in (\ref{eqn:Chu-matrix}), with determinant $f_{n-2}(x,y;u,v)$, specialized at $x=1$, $y=1$, $u=-1$ and $v=q$. Thus, by Theorem \ref{thm:Chu}, we have 
\begin{align*}
\sum_{\mathbf{x}\in \F_{n}} q^{\exc(\xx)} &= n\cdot f_{n-2}(1,1;-1,q) \\
&=n\cdot \prod_{k=0}^{n-2} \left(1+\frac{n-2}{2}+\frac{n-2-2k}{2}\sqrt{1-4q} \right) \\
&=n\cdot \prod_{k=1}^{n-1} \left(\frac{n}{2}+\frac{n-2k}{2}\sqrt{1-4q} \right) \\
&=\begin{cases}
{\displaystyle n\cdot\prod_{k=1}^{\frac{n-1}{2}} \left(\frac{n}{2}+\frac{n-2k}{2}\sqrt{1-4q} \right)\left(\frac{n}{2}-\frac{n-2k}{2}\sqrt{1-4q} \right)} &\mbox{if $n$ is odd;} \\[3ex]
{\displaystyle \frac{n^2}{2}\cdot\prod_{k=1}^{\frac{n-2}{2}} \left(\frac{n}{2}+\frac{n-2k}{2}\sqrt{1-4q} \right)\left(\frac{n}{2}-\frac{n-2k}{2}\sqrt{1-4q} \right)} &\mbox{if $n$ is even.} 
\end{cases} \\
&=\begin{cases}
{\displaystyle n\cdot\prod_{k=1}^{\frac{n-1}{2}} \big(k(n-k)+(n-2k)^2q\big)} &\mbox{if $n$ is odd;} \\[3ex]
{\displaystyle \frac{n^2}{2}\cdot\prod_{k=1}^{\frac{n-2}{2}} \big(k(n-k)+(n-2k)^2q\big)} &\mbox{if $n$ is even.} 
\end{cases}
\end{align*}

By Lemma \ref{lem:gamma-nj}, we obtain the relation
\begin{equation*}
\sum_{\mathbf{x}\in \F_{n}} q^{\exc(\xx)}=\sum_{j=0}^{\lfloor (n-1)/2\rfloor} \gamma_{n,j} q^j.
\end{equation*}
It follows that $\gamma_{n,j}$ is the number of words $\xx\in\F_n$ with $\exc(\xx)=j$. Thus, the proof of Theorem \ref{thm:gamma-positive-for-words} is completed. \qed

\smallskip
By Theorems \ref{thm:type-B-PF-NC}, \ref{thm:bijection-type-B-PF-Word} and \ref{thm:gamma-positive-for-words},  the proof of Theorem \ref{thm:gamma-positive} is completed.

\section{Concluding Remarks}
In this paper we establish a connection of parking functions of types $A$ and $B$ with the maximal chains in the lattice of noncrossing partitions of $W$ by the polynomials $M(\mathfrak{S}_{n+1},q)$ and $M(B_n,q)$.
Some potential follow-up questions may arise. There are interesting bijections between the regions of Shi arrangement and parking functions \cite{AL}, and between the regions and the maximal chains in the lattice of noncrossing partitions \cite{Stanley}.
There are other objects counted by the number $(n+1)^{n-1}$, such as Cayley trees with $n+1$ vertices, the critical states of the abelian sandpile model on the complete graph $K_{n+1}$, the regions in the Ish arrangement in $\mathbb{R}^n$ \cite{ArmstrongRhodes}, etc. 
It would be interesting to investigate the interrelationship among these combinatorial structures from the perspective of the polynomial $M(\mathfrak{S}_{n+1},q)$. 

In a study of interval partitions, Biane and Josuat-Verg\'es \cite{BJ2} generalized the polynomial $M(\mathfrak{S}_{n+1},q)$ to a multivariate generating function
\begin{equation*} 
\sum_{\ww} \wt(\ww)=\prod_{k=1}^{n-1} \big(n-k+1+kX_k\big),
\end{equation*}
summed over all maximal chains $\ww$ in $\NC(\mathfrak{S}_{n+1},c)$, where the weight $\wt$ is a monomial in the $X_k$. As quoted from \cite{BJ2}, they said ``it would be interesting to investigate the existence of similar formulas for other types.''

\section*{Acknowledgements.}
We would like to thank an anonymous referee for reading the manuscript carefully and providing helpful suggestions. 
This research was supported in part by
National Science and Technology Council, Taiwan, through grants 110-2115-M-003-011-MY3 (S.-P. Eu),  111-2115-M-153-004-MY2 (T.-S. Fu), and postdoctoral fellowship 111-2811-M-A49-537-MY2 (Y.-J. Cheng).



\begin{thebibliography}{99}

\bibitem{AL} C.~A.~Athanasiadis, S.~Linusson, A simple bijection for the regions of the Shi arrangement of hyperplanes, \textit{Discrete Math.} \textbf{204} (1999), 27--39.

\bibitem{Armstrong} D.~Armstrong, \textit{Generalized noncrossing partitions and combinatorics of Coxeter groups}, Mem. Amer. Math. Soc. 202, 2009.

\bibitem{ArmstrongRhodes}
D.~Armstrong, B.~Rhoades, The Shi arrangement and the Ish arrangement, \textit{Trans. Amer. Math. Soc.} \text{364}(3) (2012),  1509--1528.

\bibitem{Bessis} D.~Bessis, The dual braid monoid, \textit{Ann. Sci. {\'E}cole Norm. Sup.} \textbf{36}(5) (2003), 647--683.

\bibitem{Biane97} P.~Biane, Some properties of crossings and partitions, \textit{Discrete Math.} \textbf{175}(1-3) (1997), 41--53.


\bibitem{Biane} P.~Biane, Parking functions of types $A$ and $B$, \textit{Electron. J. Combin.} \textbf{9} (2002), \#N7.


\bibitem{BJ} P.~Biane, M.~Josuat-Verg\`es, Noncrossing partitions, Bruhat order and the cluster complex, \textit{Annales de l'Institut Fourier} \textbf{69}(5) (2019),  2241--2289.

\bibitem{BJ2} P.~Biane, M.~Josuat-Verg\`es, Minimal factorizations of a cycle: a multivariate generating function, \textit{Discret. Math. Theor. Comput. Sci.}, DMTCS Proc. (FPSAC 2016), 2016, pp. 239--250.

\bibitem{BB}
A.~Bj\"{o}rner, F.~Brenti, \textit{Combinatorics of Coxeter Groups}, Graduate Texts in Math., Vol. 231, Springer-Verlag, New York, 2005.

\bibitem{BW-02} T.~Brady, C.~Watt, $K(\pi, 1)$’s for Artin groups of finite type, \textit{Geom. Dedicata} \textbf{94}(1) (2002), 225--250.

\bibitem{BW} T.~Brady, C.~Watt, Non-crossing partition lattices in finite real reflection groups, \textit{Trans. Amer. Math. Soc.} \textbf{360}(4) (2008), 1983--2005.

\bibitem{Chapoton}
F.~Chapoton, Enumerative properties of generalized associahedra, \textit{S\'{e}min. Lothar. Combin.} \textbf{51} (2004), B51b.


\bibitem{Chapuy-Douvro}
G.~Chapuy, T.~Douvropoulos, Counting chains in the noncrossing partition lattice via the $W$-Laplacian, \textit{J. Algebra} \textbf{602} (2021), 381--404.

\bibitem{Chu} W.~Chu, Fibonacci polynomials and Sylvester determinant of tridiagonal matrix, \textit{Appl. Math. Comp.} \textbf{216} (2010), 1018--1023.

\bibitem{Deligne} P.~Deligne, Letter to E. Looijenga on March 9, 1974. Reprinted in the diploma thesis of P. Kluitmann,
pp. 101--111.


\bibitem{GS} I.~M.~Gessel, S.~Seo, A refinement of Cayley’s formula for trees, \textit{Electron. J. Combin.} \textbf{11}(2) (2004-6), R27.


\bibitem{Gessel-Stanley} I.~M.~Gessel, R.~Stanley, Algebraic enumeration, in: R.~Graham, M.~Gr\"otschel, L.~Lov\'asz (Eds.), \textit{Handbook of Combinatorics}, North-Holland, 1995,
pp. 1021--1061.



\bibitem{JV} M.~Josuat-Verg\`es, Refined enumeration of noncrossing chains and hook formulas, \textit{Ann. Combin.} \textbf{19} (2015), 443--460.




\bibitem{Reiner} V.~Reiner, Non-crossing partitions for classical reflection groups, \textit{Discrete Math.} \textbf{177}(1-3) (1997), 195--222.


\bibitem{RRS} V.~Reiner, V.~Ripoll, C.~Stump, On non-conjugate Coxeter elements in well-generated reflection groups, \textit{Math. Z.} \textbf{285} (2017), 1041--1062.

\bibitem{oeis} N.~J.~A.~Sloane, The On-Line Encyclopedia of Integer Sequences, published online at {\tt http://oeis.org}.


\bibitem{Stanley} R.~P.~Stanley, Parking functions and noncrossing partitions, \textit{Electron. J. Combin.} \textbf{4}(2) (1997), R20.

\bibitem{Yan} C.H.~Yan, Parking functions, in: M.~B\'ona(ed.), \textit{Handbook of Enumerative Combinatorics}, Discrete Math. Appl., CRC Press, Boca Raton, 2015, pp. 835--893.


\end{thebibliography}
\end{document}